\documentclass{article}

\usepackage{enumerate}
\usepackage{amsthm}
\usepackage{amsbsy}
\usepackage{amscd}
\usepackage{amsmath,amscd,amssymb,epsfig,subfig}
\usepackage[all]{xy}
\usepackage{verbatim}
\usepackage{tikz}
\theoremstyle{plain}
\usepackage{etex}
\usepackage{graphicx}
\newtheorem{theorem}{Theorem}
\newtheorem{corollary}{Corollary}

\newtheorem{proposition}{Proposition}

\newtheorem{remark}{Remark}

\allowdisplaybreaks
\begin{document}

\date{\today}

\title{On the Ambartzumian-Pleijel identity in hyperbolic geometry}
  \author{Binbin XU \\
\small \em Institut Fourier\\
\small \em 100 Rue Des Maths \\
\small \em 38402 Saint-Martin-d'H\`eres cedex, France\\
\small \em e-mail: {\tt binbin.xu@ujf-grenoble.fr}\\
}

\maketitle
\begin{abstract}
We describe a hyperbolic version of the Ambartzumian-Pleijel identity. We use this identity to prove the hyperbolic Crofton formula and the hyperbolic isoperimetric inequality. This identity also provides a way to compute the chord length distribution for an ideal polygon in the hyperbolic plane. The analogous results for a maximally symmetric, simply connected, $2-$dimensional Riemannian manifold with constant sectional curvature are given in the end.
\end{abstract}

\section{Introduction}
In \cite{plei}, Pleijel discovered a family of identities associated with isoperimetric inequalities for planar convex domains with $C^1$ boundary. In \cite{amba1}, by generalizing these identities, Ambartzumian gave the Pleijel identity for the Euclidean plane. In \cite{amba}, Ambartzumian gave a combinatorial proof of the Pleijel identity. Moreover, he proved a general version of Pleijel identity for convex compact polygonal planar domains which we call the Ambartzumian-Pleijel identity. Later in \cite{cabo}, Cabo gave another approach to this identity via Stokes' theorem. In \cite{amba}, Ambartzumian also pointed out that these two identities can be used to find chord length distributions for planar convex domains using $\delta-$formalism. 

More precisely, let $\mathbb{E}$ denote the Euclidean plane. Let $\mathcal{G}_\mathbb{E}$ denote the set of geodesics in $\mathbb{E}$ and let $\mu_\mathbb{E}$ denote the measure on $\mathcal{G}_\mathbb{E}$ invariant under the Euclidean motions. Let $D$ be a compact convex domain in $\mathbb{E}$ and consider the subset $\mathcal{G}_D$ of $\mathcal{G}_\mathbb{E}$ consisting of all geodesics intersecting $D$. For each $\gamma\in \mathcal{G}_D$, the intersection $\gamma\cap D$ is called a \textit{chord} of $\gamma$ with respect to $D$. Let $\rho_D(\gamma)$ denote the chord length of $\gamma$ with respect to $D$. 

Assume the boundary $\partial D$ of $D$ is $C^1$. Let $\alpha_1(\gamma)$ and $\alpha_2(\gamma)$ be the two angles between the boundary $\partial D$ and the chord of $\gamma$ lying on the same side of $\gamma$.
\begin{center}
\includegraphics[scale=0.7]{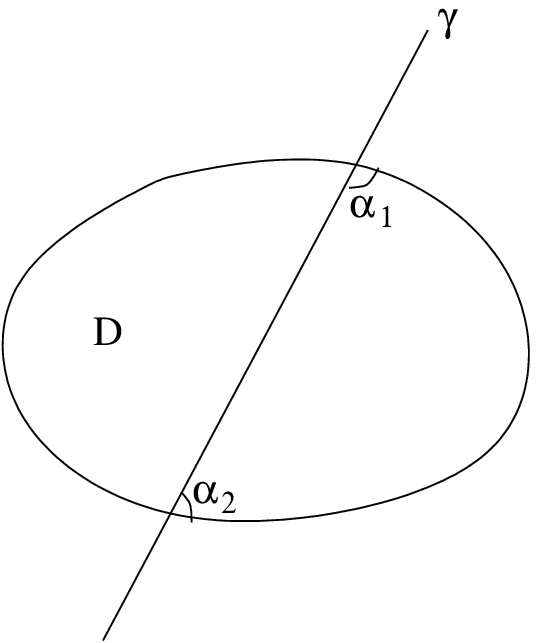}
\end{center}
With this notation, the Pleijel identity is as follows:
\begin{equation*}
 \int_{\mathcal{G}_{D}}(f\circ\rho_D)\,{\rm d}\mu_\mathbb{E}=\int_{\mathcal{G}_{D}}(f'\circ\rho_D)\rho_D\cot\alpha_1\cot\alpha_2 \,{\rm d}\mu_\mathbb{E},
\end{equation*}
where $f\in C^1(\mathbb{R},\mathbb{R})$ with $f(0)=0$.
 
This identity implies the isoperimetric inequality $L(\partial D)^2-4\pi A(D)\ge 0$, where $L(\partial D)$ is the boundary length of $D$ and $A(D)$ is the area of $D$.

Instead of being $C^1$, now assume that $\partial D$ is a polygon consisting of $n$ geodesic edges $a_1,\dots,a_n$. The Ambartzumian-Pleijel identity for $D$ and any $C^1-$function $f$ is as follows:
\begin{equation*}
 \int_{\mathcal{G}_{D}}(f\circ\rho_D)\,{\rm d}\mu_\mathbb{E}=\int_{\mathcal{G}_{D}}(f'\circ\rho_D)\rho_D\cot\alpha_1\cot\alpha_2 \,{\rm d}\mu_\mathbb{E}+\sum^n_{i=1}\int^{|a_i|}_0f(x)\,{\rm d}x,
\end{equation*}
where $|a_j|$ is the length of $a_j$, and ${\rm d}x$ is the Euclidean length element on $\mathbb{R}$. 

In this paper, we prove the hyperbolic counterpart of the Ambartzumian-Pleijel identity. Let $D$ be a compact convex domain in the hyperbolic plane $\mathbb{H}$ with geodesic polygon boundary whose edges are denoted by $a_1,\dots,a_n$. Let $\mathcal{G}_{\mathbb{H}}$ denote the set of geodesics in $\mathbb{H}$ and $\mu$ denote the Liouville measure on $\mathcal{G}_{\mathbb{H}}$. Let $\mathcal{G}_D$ be the subset of $\mathcal{G}_{\mathbb{H}}$ consisting of geodesics intersecting $D$. Our main result is:
\begin{theorem}\label{1}
 Let $f$ be in $C^1(\mathbb{R};\mathbb{R})$. Then we have the following hyperbolic version of the Ambartzumian-Pleijel identity:
\begin{equation*}
 \int_{\mathcal{G}_{D}}(f\circ\rho_D)\,{\rm d}\mu=\int_{\mathcal{G}_{D}}(f'\circ\rho_D)\sinh\rho_D\cot\alpha_1\cot\alpha_2 \,{\rm d}\mu+\frac{1}{2}\sum^n_{i=1}\int^{|a_i|}_0f(x)\,{\rm d}x,
\end{equation*}
where ${\rm d}x$ is the length element on $\mathbb{R}$, and $|a_i|$ is the hyperbolic length of the i-th boundary segment $a_i$ of $D$.
\end{theorem}

By using polygonal domains to approximate hyperbolic convex compact domains with $C^1-$boundary, we can prove the hyperbolic version of the Pleijel identity:

\begin{theorem}\label{2}
 Let $f$ be in $C^1(\mathbb{R};\mathbb{R})$. Suppose that $\partial D$ is $C^1$. With the same notation as in Theorem \ref{1}, we have the following identity:
\begin{equation*}
 \int_{\mathcal{G}_{D}}(f\circ\rho_D)\,{\rm d}\mu=\int_{\mathcal{G}_{D}}(f'\circ\rho_D)\sinh\rho_D\cot\alpha_1\cot\alpha_2 \,{\rm d}\mu+\frac{1}{2}f(0)L(\partial D).
\end{equation*}
In particular if $f(0)=0$, then we have the hyperbolic version of the Pleijel identity.
\end{theorem}

By choosing $f$ carefuller, we prove the following two applications of Theorems \ref{1} and \ref{2}: 

\begin{corollary}\label{c1}
Let $D$ be a convex compact domain in $\mathbb{H}$ with polygonal boundary or $C^1$ boundary, the Liouville measure of $\mathcal{G}_D$ is one half of the length of the boundary of $D$.
\end{corollary}
\begin{corollary}\label{c2}
Let $\partial D$ be $C^1$. Then we have the hyperbolic isoperimetric inequality:
\begin{equation*} 
L(\partial D)^2\ge4\pi A(D)+A(D)^2,
\end{equation*}
where the equality holds if and only if $D$ is a disk in $\mathbb{H}$.
\end{corollary}
A priori, Theorems \ref{1} and \ref{2} only hold for a compact domain $D$. But the strategy of the proof can be extended to the non compact case so that we are able to compute the chord length distributions for an ideal triangle and an ideal quadrilateral in $\mathbb{H}$ as follows:
\begin{corollary}\label{c3}
 Let $T$ be an ideal triangle in $\mathbb{H}$ and $\mu$ be the Liouville measure on $\mathcal{G}_{\mathbb{H}}$. The chord length distribution ${\rm d}M_T=(\rho_T)_*{\rm d}\mu$ is given by:
 \begin{displaymath}
  {\rm d}M_T=\frac{3\rho\, {\rm d}\rho}{\sinh^2\rho}.
 \end{displaymath}
\end{corollary}
\begin{corollary}\label{c4}
 Let $Q$ be an ideal quadrilateral in $\mathbb{H}$ and $\mu$ be the Liouville measure on $\mathcal{G}_{\mathbb{H}}$. Let $\gamma_1,\dots,\gamma_4$ be the 4 edges of $Q$ ordered counter-clockwise. The chord length distribution ${\rm d}M_Q=(\rho_Q)_*{\rm d}\mu$ is given by:
 \begin{displaymath}
  {\rm d}M_Q=\frac{12\rho\,{\rm d}\rho}{\sinh^2\rho}+{\rm d}M_{13}+{\rm d}M_{24},
 \end{displaymath}
where ${\rm d}M_{13}$ is the chord length distribution with respect to $\gamma_1$ and $\gamma_3$ and satisfies:
\begin{equation*}
\int_0^\rho{\rm d}M_{13}=\frac{1}{2}\int_{[\eta]}\frac{\cot\alpha_1(\rho,\eta)\cot\alpha_3(\rho,\eta)\sinh\rho\cosh w(\rho,\eta)}{\sinh\rho_1(\rho,\eta)\cot\alpha_1(\rho,\eta)+\sinh\rho_3(\rho,\eta)\cot\alpha_3(\rho,\eta)} \,{\rm d}\eta,
\end{equation*}
and ${\rm d}M_{24}$ is the chord length distribution with respect to $\gamma_2$ and $\gamma_4$ and satisfies:
\begin{equation*}
\int_0^\rho{\rm d}M_{24}=\frac{1}{2}\int_{[\eta]}\frac{\cot\alpha_2(\rho,\eta)\cot\alpha_4(\rho,\eta)\sinh\rho\cosh w(\rho,\eta)}{\sinh\rho_2(\rho,\eta)\cot\alpha_2(\rho,\eta)+\sinh\rho_4(\rho,\eta)\cot\alpha_4(\rho,\eta)} \,{\rm d}\eta,
\end{equation*} 
where $\eta$ is the angle parameter in the polar parametrization of the set of geodesics in $\mathbb{H}$ introduced later.
\end{corollary}
\begin{remark}
Corollary \ref{c1} is called the Crofton's formula.

Corollary \ref{c3} and Corollary \ref{c4} are the principle motivations of this paper. Corollary \ref{c3} was first proved by Bridgeman and Dumas in \cite{martin}. In that paper they considered the oriented geodesics which yields a difference by a factor $2$. In \cite{martin2}, instead of the Liouville measure, Bridgeman considered the volume form of the unit tangent bundle of hyperbolic space and described the length distribution using the Rogers dilogarithm functions.
\end{remark}

Another observation about the proof of the Ambartzumian-Pleijel identity is that as passing from the Euclidean case to the hyperbolic case, the only thing changed is the trigonometric functions. As a result of this observation, by using the general trigonometric functions, we show that our two theorem and their first two corollaries can be generalized for a maximally symmetric, simply connected, $2-$dimensional Riemannian manifold with constant sectional curvature $K\in\mathbb{R}$, denoted by $\mathbb{X}_K$.

The organization of this paper is the following: in the section $2$ and $3$, we give several formulas of the hyperbolic metric, its volume form and the Liouville measure; in section $4$, we prove the main theorems; in section $5$, all corollaries are proved; in the last section, we give our result for the $\mathbb{X}_K$ case.

\section*{Acknowledgement}
The author is greatly indebted to Greg McShane for suggesting this problem and for the useful discussions. The author thanks Andrew Yarmola for the useful discussions.

\section{The hyperbolic plane and its coordinate systems}
In this section we describe several coordinate systems on the upper half plane:
\begin{equation*}
 \mathbb{H}=\{z=x+iy\in\mathbb{C}: y\ge0\}.
\end{equation*}
Moreover we give the associated volume form for each coordinate system.

\subsection{Cartesian coordinate system}
This is the most familiar coordinate system of the upper half plane model. The coordinates of $\mathbb{C}$ induce coordinates on $\mathbb{H}$. 
\begin{center}
 \includegraphics[scale=0.6]{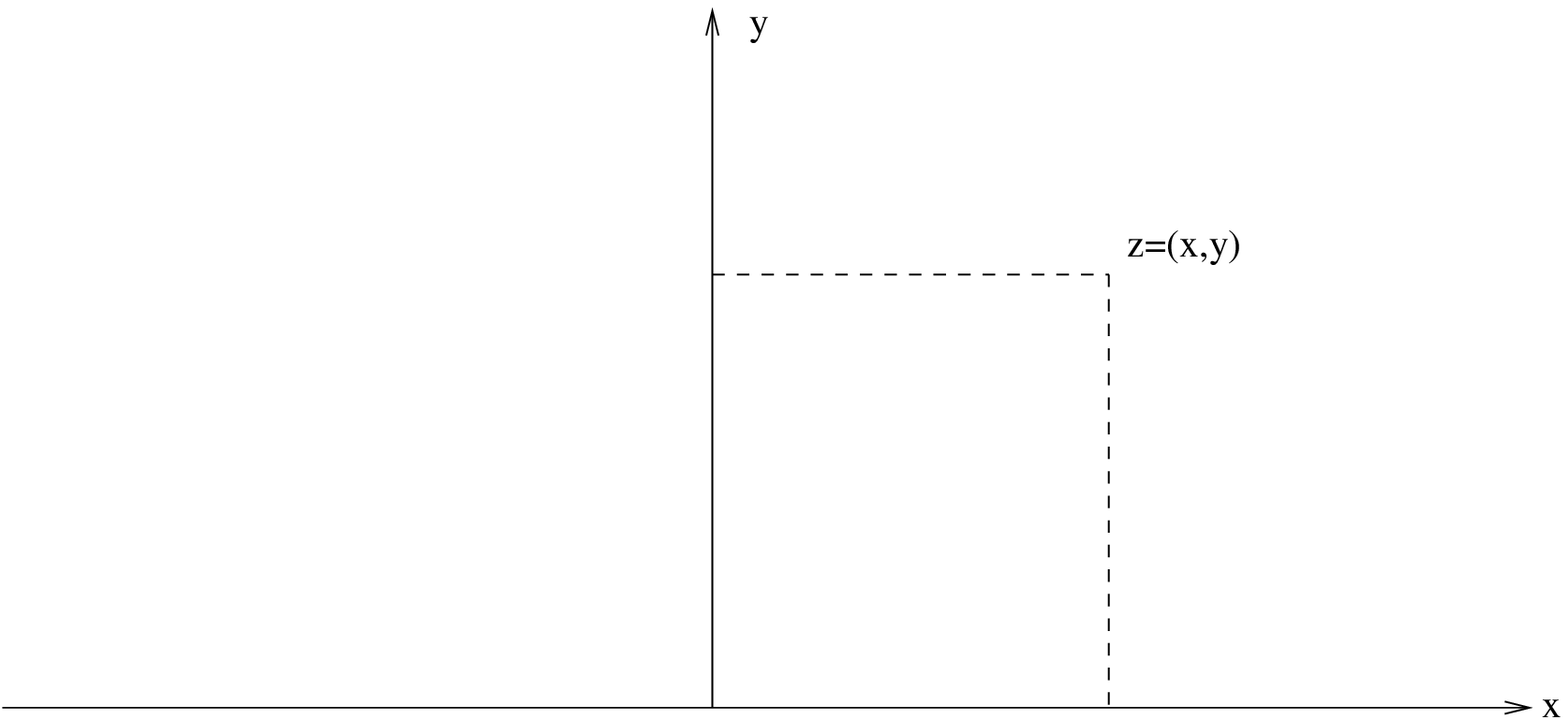}
\end{center}
By using these coordinates, the hyperbolic metric ${\rm d}s$ has the following expression:
\begin{equation}
 {\rm d}s=\frac{\sqrt{({\rm d}x)^2+({\rm d}y)^2}}{y}.
\end{equation}
and the associated hyperbolic volume form ${\rm dVol}$  is given by:
\begin{equation}
 {\rm dVol}=\frac{{\rm d}x{\rm d}y}{y^2}.
\end{equation}

\subsection{Polar coordinate system}
Generally speaking, this coordinate system is the pull back of the polar coordinate system of the Poincar\'e disk model $\mathbb{D}$ by the Cayley transformation:
\begin{equation}\label{cayley}
 \omega(z)=\frac{iz+1}{z+i},
\end{equation}

Let $z=(R,\theta)$ denote the polar coordinate system of $\mathbb{C}$ where $R$ is the radius and $\theta$ is the angle. Under this coordinates, the hyperbolic metric on $\mathbb{D}$ is given by:
\begin{equation}\label{disk}
{\rm d}s=\frac{2\sqrt{R^2({\rm d}\theta)^2+({\rm d}R)^2}}{1-R^2},
\end{equation}
and the associated volume form is:
\begin{equation*}
{\rm dVol}=\frac{4R\,{\rm d}R{\rm d}\theta}{(1-R^2)^2}.
\end{equation*}

By replacing the Euclidean radius $R$ by the hyperbolic radius $r$, we obtain a new coordinate system $(r,\theta)$ for $\mathbb{D}$. The relation between $r$ and $R$ can be computed by integrating Formula (\ref{disk}) along the geodesic from the origin of $\mathbb{D}$ to $(R,\theta)$:
\begin{equation*}
 r=\ln\frac{1+R}{1-R},
\end{equation*}
or equivalently:
\begin{equation*}
 \tanh \frac{r}{2}=R.
\end{equation*}

By changing variables, the hyperbolic metric under the new polar coordinates is:
\begin{equation*}
 {\rm d}s=\sqrt{\sinh^2 r({\rm d}\theta)^2+({\rm d}r)^2},
\end{equation*}
and its volume form is the following:
\begin{equation*}
 {\rm dVol}=\sinh r\,{\rm d}r{\rm d}\theta.
\end{equation*}

The pullback of this coordinate system of $\mathbb{D}$ by (\ref{cayley}) induces a coordinate system on $\mathbb{H}$ which we call the polar coordinate system of $\mathbb{H}$. 
\begin{center}
 \includegraphics[scale=0.6]{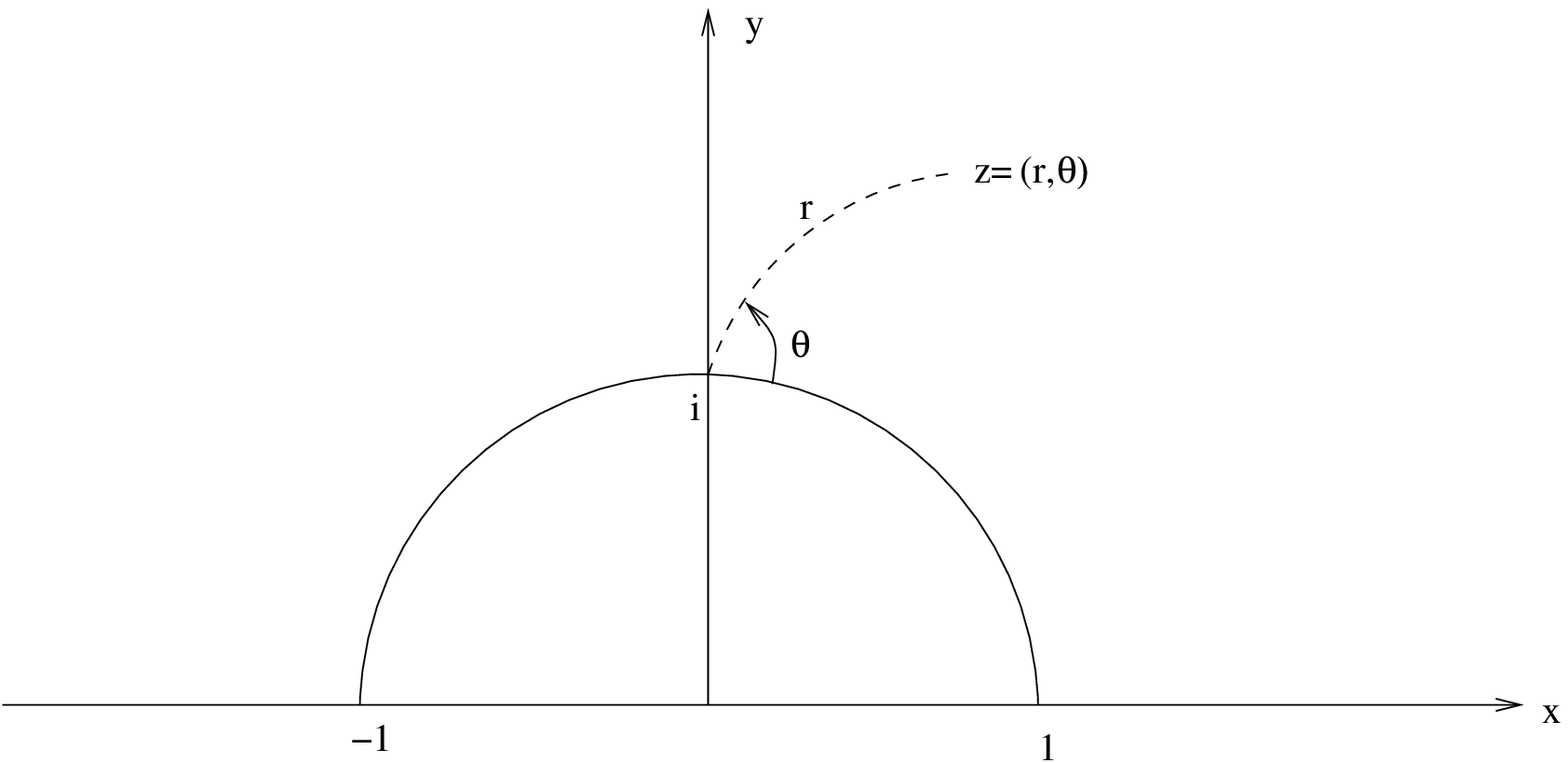}
\end{center}

\begin{remark}
By pulling back by the composition of the Cayley transformation (\ref{cayley}) with an element $A\in{\rm PSL}(2,\mathbb{R})$, we can define a new polar coordinate system $(r',\theta')$ for an arbitrary fixed pair $(z_0,\gamma_0)$ where $z_0$ is a point in $\mathbb{H}$ and $\gamma_0$ is a half-geodesic starting at $z_0$ such that $z_0$ has coordinates $(0,0)$ as the origin and the points on $\gamma_0$ have coordinates $(r,0)$. Let $z\in\mathbb{H}$. It has the coordinates $(r,\theta)$ and $(r',\theta')$ under the two different coordinate systems. By changing variables, we can verify that:
\begin{equation*}
 {\rm d}s(z)=\sqrt{\sinh^2 r({\rm d}\theta)^2+({\rm d}r)^2}=\sqrt{\sinh^2 r'({\rm d}\theta')^2+({\rm d}r')^2},
\end{equation*}
and
\begin{equation*}
 {\rm dVol}(z)=\sinh r\,{\rm d}r{\rm d}\theta=\sinh r'\,{\rm d}r'{\rm d}\theta'.
\end{equation*}
\end{remark}
\subsection{Rectangular coordinate system}
This coordinate is similar to the Cartesian coordinates for $\mathbb{C}$. Generally speaking we fix two geodesics as the horizontal axis and the vertical axis in $\mathbb{H}$ which are orthogonal to each other. Each point in $\mathbb{H}$ is described by a vertical coordinate and a horizontal coordinate. 

More precisely, let $\gamma_v$ be the geodesic ending at $0$ and $\infty$ with $\infty$ to be the positive end and let $\gamma_h$ be the geodesic ending at $1$ and $-1$ with $1$ to be the positive end. Set the intersection point $z=i$ to be the origin for both $\gamma_v$ and $\gamma_h$. Then we can parametrize both $\gamma_v$ and $\gamma_h$ by considering the directed hyperbolic distance from their origin. Denote by $p$ the parameter for $\gamma_v$ and by $q$ the parameter for $\gamma_h$. We use parameter $p$ and parameter $q$ to construct the rectangular coordinate system in the following way: the points on $\gamma_v$ have the rectangular coordinates $(p,0)$ and the points on $\gamma_h$ have the rectangular coordinates $(0,q)$; the horizontal line $p=p_0$ in this coordinate system is the geodesic orthogonal to $\gamma_v$ at the point $(p_0,0)$ and the vertical line $q=q_0$ consists of the points in $\mathbb{H}$ have hyperbolic distance $|q_0|$ to $\gamma_h$ and on the same side as $(0,q_0)$ with respect to $\gamma_v$. 
\begin{center}
 \includegraphics[scale=0.6]{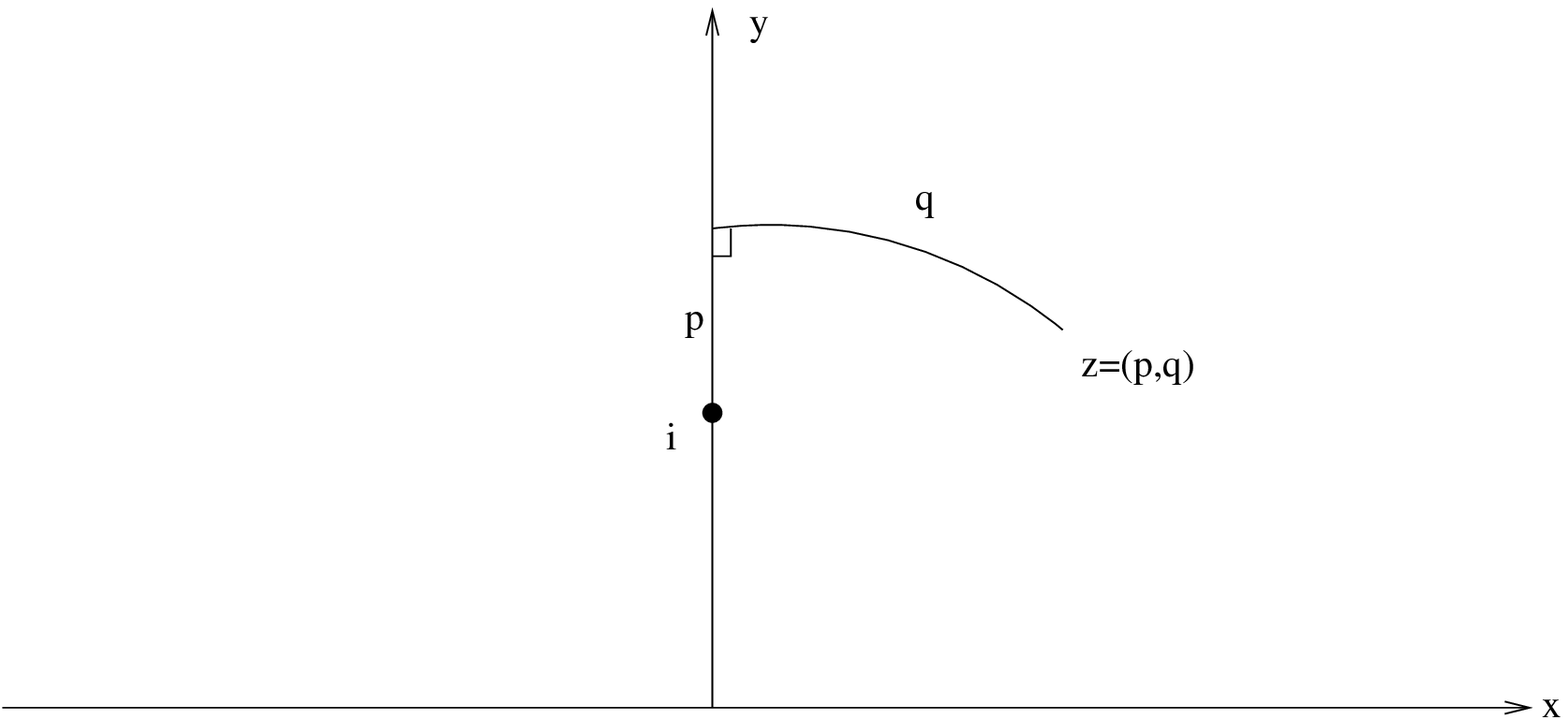}
\end{center}

By considering the usual polar coordinate of $\mathbb{C}$, we obtain the following relation between the Cartesian coordinates $(x,y)$ and the rectangular coordinates $(p,q)$ of $\mathbb{H}$:
\begin{eqnarray*}
 &&x=e^p\tanh q,\\
 &&y=\frac{e^p}{\cosh q}.
\end{eqnarray*}
By changing variables, we get the expression of the hyperbolic metric and its volume form using the parameters $p$ and $q$ as following:
\begin{eqnarray}
 &&{\rm d}s=\sqrt{\cosh^2q({\rm d}p)^2+({\rm d}q)^2},\\
 &&{\rm dVol}=\cosh q\,{\rm d}p{\rm d}q.
\end{eqnarray}

\begin{remark}
Choose another pair of oriented geodesics $(\gamma_v',\gamma_h')$ orthogonal to each other. By the same construction as above, we obtain another rectangular coordinate system $(p',q')$. Let $z\in \mathbb{H}$. Under two distinct rectangular coordinate system, it has two pairs of coordinates $(p,q)$ and $(p',q')$. We can verify that:
\begin{equation*}
{\rm d}s(z)=\sqrt{\cosh^2q({\rm d}p)^2+({\rm d}q)^2}=\sqrt{\cosh^2q'({\rm d}p')^2+({\rm d}q')^2}, 
\end{equation*}
and
\begin{equation*}
 {\rm dVol}(z)=\cosh q'{\rm d}p'{\rm d}q'=\cosh q\,{\rm d}p{\rm d}q.
\end{equation*}
\end{remark}

\section{Liouville Measure}
The group of orientation preserving isometries of $\mathbb{H}$ is isomorphic to ${\rm PSL}(2,\mathbb{R})$ and the action is given by M\"obius transformation. There is an unique measure $\mu$, up to scalar multiplication, on $\mathcal{G}_{\mathbb{H}}$ invariant under the action of ${\rm PSL}(2,\mathbb{R})$, namely the \textit{Liouville measure}. In this section, we give several parametrizations of $\mathcal{G}_{\mathbb{H}}$ and the associated expressions of $\mu$.

\subsection{Parametrization by using the boundary of $\mathbb{H}$}
A geodesic in $\mathbb{H}$ is uniquely determined by its end points in $\partial\mathbb{H}$. By this correspondence $\mathcal{G}_{\mathbb{H}}$ can be identified with the set $((\partial \mathbb{H}\times \partial \mathbb{H})\setminus \Delta)/\mathbb{Z}_2$ where $\Delta$ is the diagonal of $\partial \mathbb{H}\times \partial \mathbb{H}$ and the $\mathbb{Z}_2$ action exchanges the two end points of a geodesic. In the upper half plane model the boundary $\partial \mathbb{H}$ can be identified with $\mathbb{R}\cup\{\infty\}$. 
\begin{center}
 \includegraphics[scale=0.6]{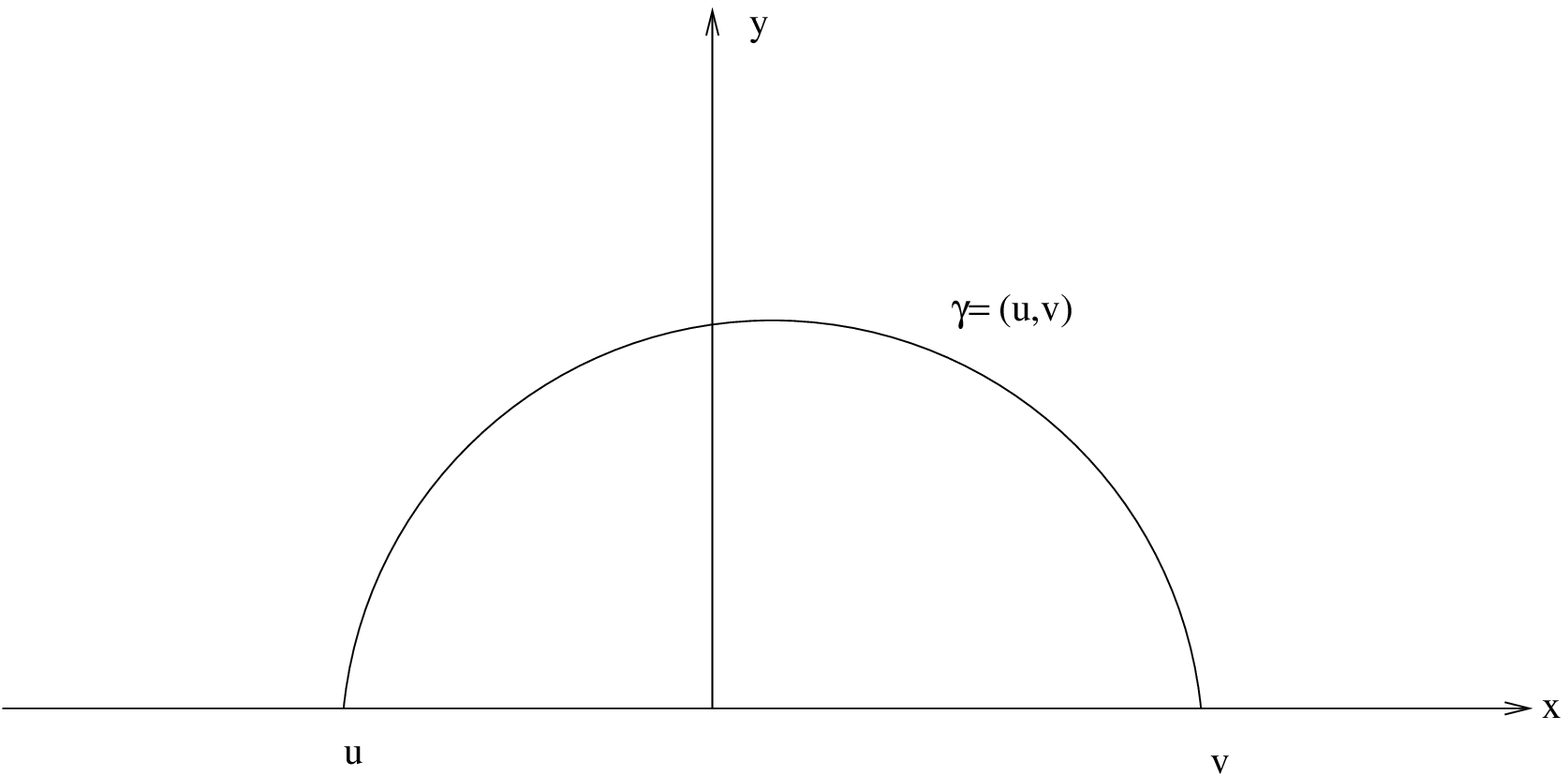}
\end{center}

With this parametrization, the Liouville measure ${\rm d}\mu$ at a geodesic $\gamma=(u,v)$ is given by:
\begin{displaymath}
 \rm{d}\mu(u,v)=\frac{{\rm d}u{\rm d}v}{|u-v|^2}.
\end{displaymath}

Suppose that $[a,b]$ and $[c,d]$ are two disjoint intervals in $\mathbb{R}\cup\{\infty\}$.  By integrating ${\rm d}\mu$, the Liouville measure of $[a,b]\times[c,d]$ is given in term of a cross-ratio:
\begin{displaymath}
 \mu([a,b]\times[c,d])=\left\lvert\log\left\lvert\frac{(a-c)(b-d)}{(a-d)(b-c)}\right\lvert\right\lvert.
\end{displaymath}

\subsection{Local parametrization by using oriented geodesics}
Let $\gamma_1$ be an oriented geodesic in $\mathbb{H}$. Let $\mathcal{G}_{\gamma_1}$ denote the set of geodesics intersecting $\gamma_1$. We fix a point of $\gamma_1$ to be the origin and fix an orientation on $\gamma_1$. We parametrize $\gamma_1$ by considering the directed hyperbolic distance from the origin. Let $l_1$ denote the parameter on $\gamma_1$. A geodesic $\gamma\in \mathcal{G}_{\gamma_1}$ is determined uniquely by the position of its intersection point $l_1$ and its angle of intersection $\alpha_1$ with $\gamma_1$. The intersection angle $\alpha_1$ is measured from $\gamma_1$ to $\gamma$ counter-clockwise. 
\begin{center}
 \includegraphics[scale=0.6]{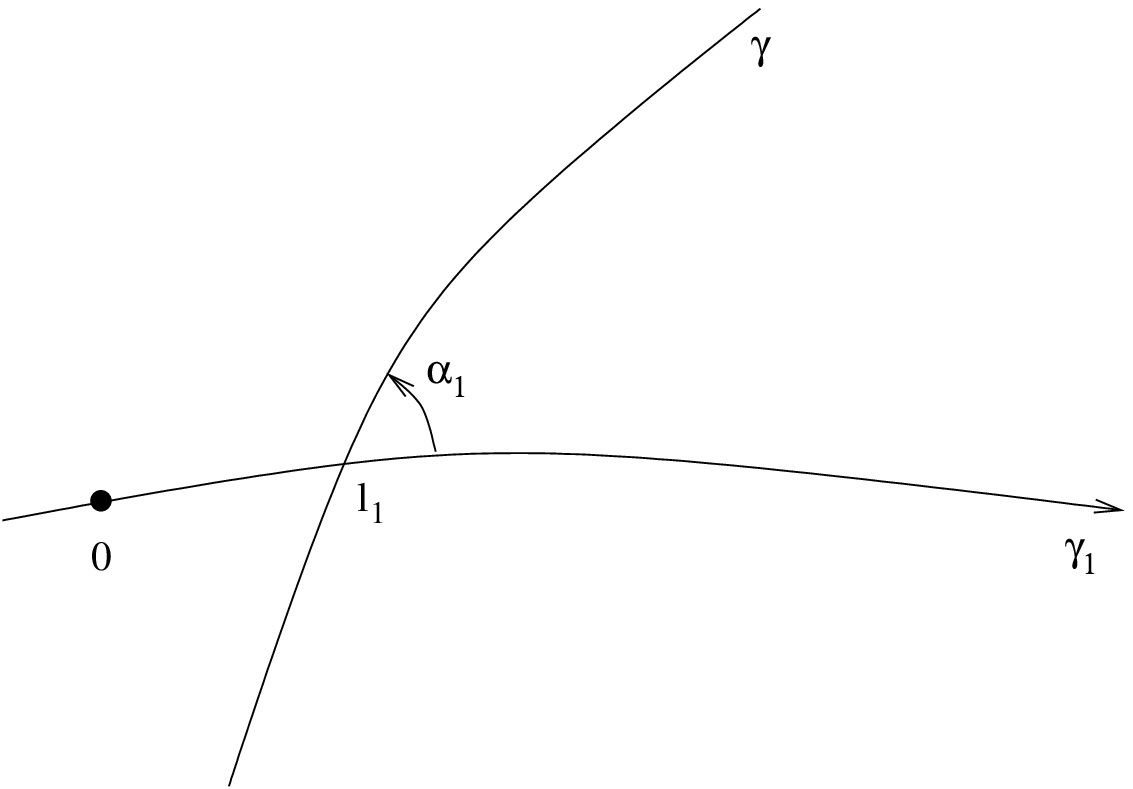}
\end{center}
On the set $\mathcal{G}_{\gamma_1}$ we have that:
\begin{displaymath}
 {\rm d}\mu(l_1,\alpha_1)=F(l_1,\alpha_1)\,{\rm d}l_1{\rm d}\alpha_1.
\end{displaymath}

\begin{proposition}
 The density $F(l_1,\alpha_1)$ is independent of the parameter $l_1$.
\end{proposition}
\begin{proof}
Fixing an angle $\alpha_1$, we need to prove that for any two distinct real numbers $l_1$ and $l_1'$, we have 
\begin{equation*}
F(l_1,\alpha_1)=F(l_1',\alpha_1).
\end{equation*}

For any pair of distinct real numbers $l_1$ and $l_1'$, there exists a hyperbolic element $g$ of ${\rm PSL}(2,\mathbb{R})$ fixing the end points of $\gamma_1$ such that $g$ sends the geodesic $\gamma=(l_1,\alpha_1)$ to $\gamma'=(l_1',\alpha_1)$. Let $L(g)$ be the translation length of $g$. Then $l_1'=l_1+L(g)$. The invariance of Liouville measure implies the following equality:
\begin{displaymath}
 F(l_1,\alpha_1)\,{\rm d}l_1{\rm d}\alpha_1=F(l_1',\alpha_1)\,{\rm d}l_1'{\rm d}\alpha_1.
\end{displaymath}

By changing variables, we obtain: 

\begin{displaymath}
 F(l_1,\alpha_1)\,{\rm d}l_1{\rm d}\alpha_1=F(l_1+L(g),\alpha)\,{\rm d}l_1{\rm d}\alpha_1,
\end{displaymath}
which implies that 
\begin{displaymath}
 F(l_1,\alpha)=F(l_1+L(\gamma),\alpha).
\end{displaymath}
\end{proof}
As a consequence of the above, we write $F(\alpha_1)$ short for $F(l_1,\alpha_1)$. To compute $F(\alpha_1)$, we choose another oriented geodesic $\gamma_2$ different from $\gamma_1$. By an analogous construction, we can define the parameters $(l_2,\alpha_2)$ for the set $\mathcal{G}_{g_2}$ of geodesics intersecting $\gamma_2$. 
\begin{remark}
The geodesic $\gamma_2$ can be any geodesic in $\mathcal{G}_\mathbb{H}$ other than $\gamma_1$, not necessarily disjoint with $\gamma_1$.

To simplify the computation, the definition of $\alpha_2$ is slightly different from $\alpha_1$. It is the angle measured from $\gamma_2$ to $\gamma$ clockwise instead of counter-clockwise.

Notice that to have both the parameters $(l_1,\alpha_1)$ and the parameters $(l_2,\alpha_2)$, the geodesic $\gamma$ need to intersect both $\gamma_1$ and $\gamma_2$. 
\end{remark}
\begin{center}
 \includegraphics[scale=0.6]{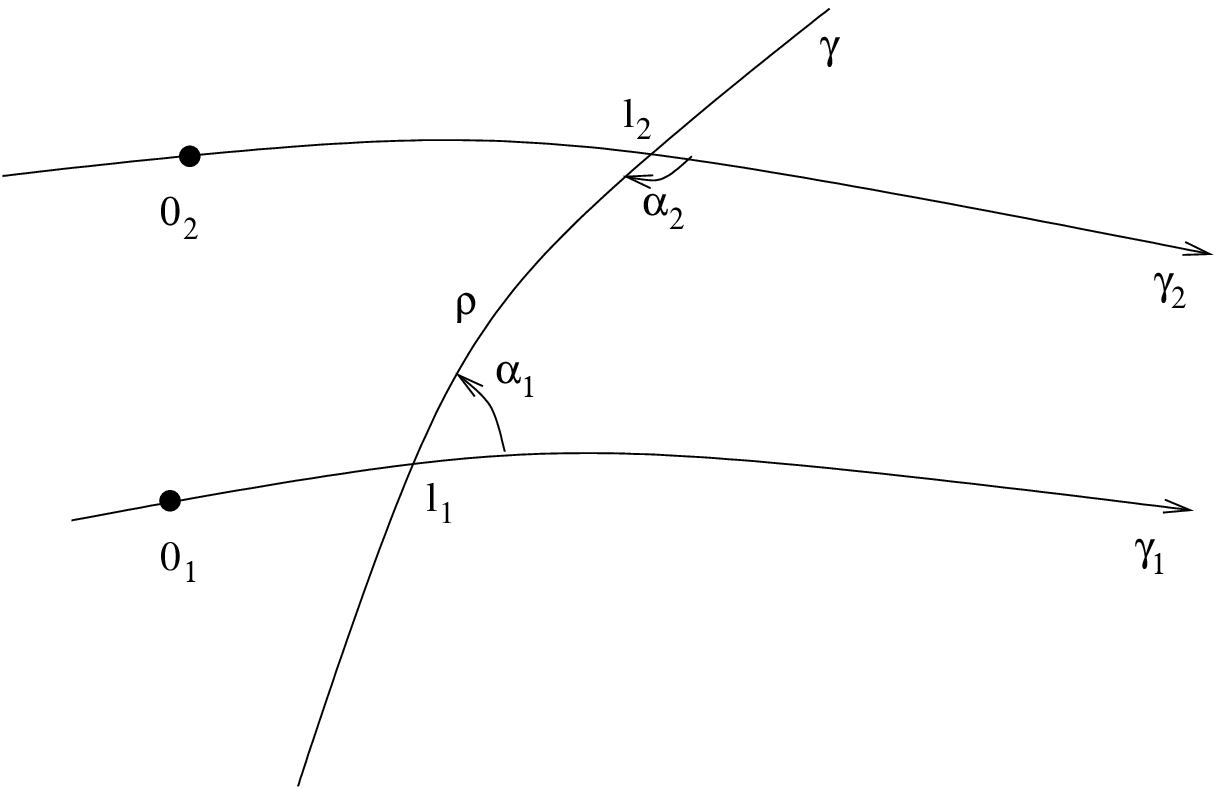}
\end{center}

By the same argument as above, we have:
\begin{displaymath}
 {\rm d}\mu(l_2,\alpha_2)=F(\pi-\alpha_2)\,{\rm d}l_2{\rm d}\alpha_2.
\end{displaymath}

The parameters $(l_1,\alpha_1)$ are evidently functions of $(l_2,\alpha_2)$ and vice-versa. By hyperbolic trigonometry, we have the following expressions for partial derivatives:
\begin{equation*}
 \frac{\partial\alpha_1}{\partial l_2}=\sigma_1(\gamma)\frac{\sin\alpha_2}{\sinh\rho(\gamma)},
\end{equation*}
and 
\begin{equation*}
 \frac{\partial\alpha_2}{\partial l_1}=\sigma_2(\gamma)\frac{\sin\alpha_1}{\sinh\rho(\gamma)},
\end{equation*}
where $\rho(\gamma)$ is the chord length of $\gamma$ with respect to $\gamma_1$ and $\gamma_2$, and the values of $\sigma_1(\gamma)$ and $\sigma_2(\gamma)$ depend on the relative position of the chord of $h$ with respect to $\gamma_1$ and $\gamma_2$: 
\begin{enumerate}[(i)]
 \item if the chord is on the left of $\gamma_1$ and left of $\gamma_2$, then $\sigma_1(\gamma)=-1$ and $\sigma_2(\gamma)=1$ ;
 \item if the chord is on the left of $\gamma_1$ and right of $\gamma_2$, then $\sigma_1(\gamma)=1$ and $\sigma_2(\gamma)=1$ ;
 \item if the chord is on the right of $\gamma_1$ and left of $\gamma_2$, then $\sigma_1(\gamma)=-1$ and $\sigma_2(\gamma)=-1$ ;
 \item if the chord is on the right of $\gamma_1$ and right of $\gamma_2$, then $\sigma_1(\gamma)=1$ and $\sigma_2(\gamma)=-1$.
\end{enumerate}

Now changing variables yields:
\begin{equation*}
 {\rm d}\mu=F(\alpha_1)\,{\rm d}l_1{\rm d}\alpha_1=\frac{F(\alpha_1)\sin\alpha_2}{\sinh\rho}\,{\rm d}l_1{\rm d}l_2,
\end{equation*}
and
\begin{equation*}
 {\rm d}\mu=F(\pi-\alpha_2)\,{\rm d}l_2{\rm d}\alpha_2=\frac{F(\pi-\alpha_2)\sin\alpha_1}{\sinh\rho}\,{\rm d}l_1{\rm d}l_2.
\end{equation*}

By inspection, the function $F$ has the form:
\begin{equation}\label{c}
 F(\alpha)=c\sin\alpha.
\end{equation}
where $c$ is a constant positive real number.

Taking $c=1/2$ we obtain exactly the Liouville measure that we gave in the former section:
\begin{equation}\label{Bon}
 {\rm d}\mu=\frac{1}{2}\sin\alpha\,{\rm d}l{\rm d}\alpha.
\end{equation}

The above computation also yields another local expression for $\mu$:
\begin{equation}\label{local}
 {\rm d}\mu=\frac{\sin\alpha_1\sin\alpha_2}{2\sinh\rho}\,{\rm d}l_1{\rm d}l_2.
\end{equation}
\begin{remark}
 The Formula (\ref{Bon}) was previously obtained by Bonahon in the appendix of \cite{Bonahon} using the Poincar\'e disc model $\mathbb{D}$.
\end{remark}

\subsection{Polar parametrization}
The polar parameters for a geodesic $\gamma\in \mathcal{G}$ are a pair $(w,\eta)$. The parameter $w$ is the hyperbolic distance from the point $i\in\mathbb{C}$ to $\gamma$. The parameter $\eta$ is the angle between two geodesics: one is the geodesic passing $i$ and orthogonal to $\gamma$; the other one is the geodesic whose end points are $1$ and $-1$. It is measured from the latter to the former counter-clockwise.
\begin{center}
 \includegraphics[scale=0.6]{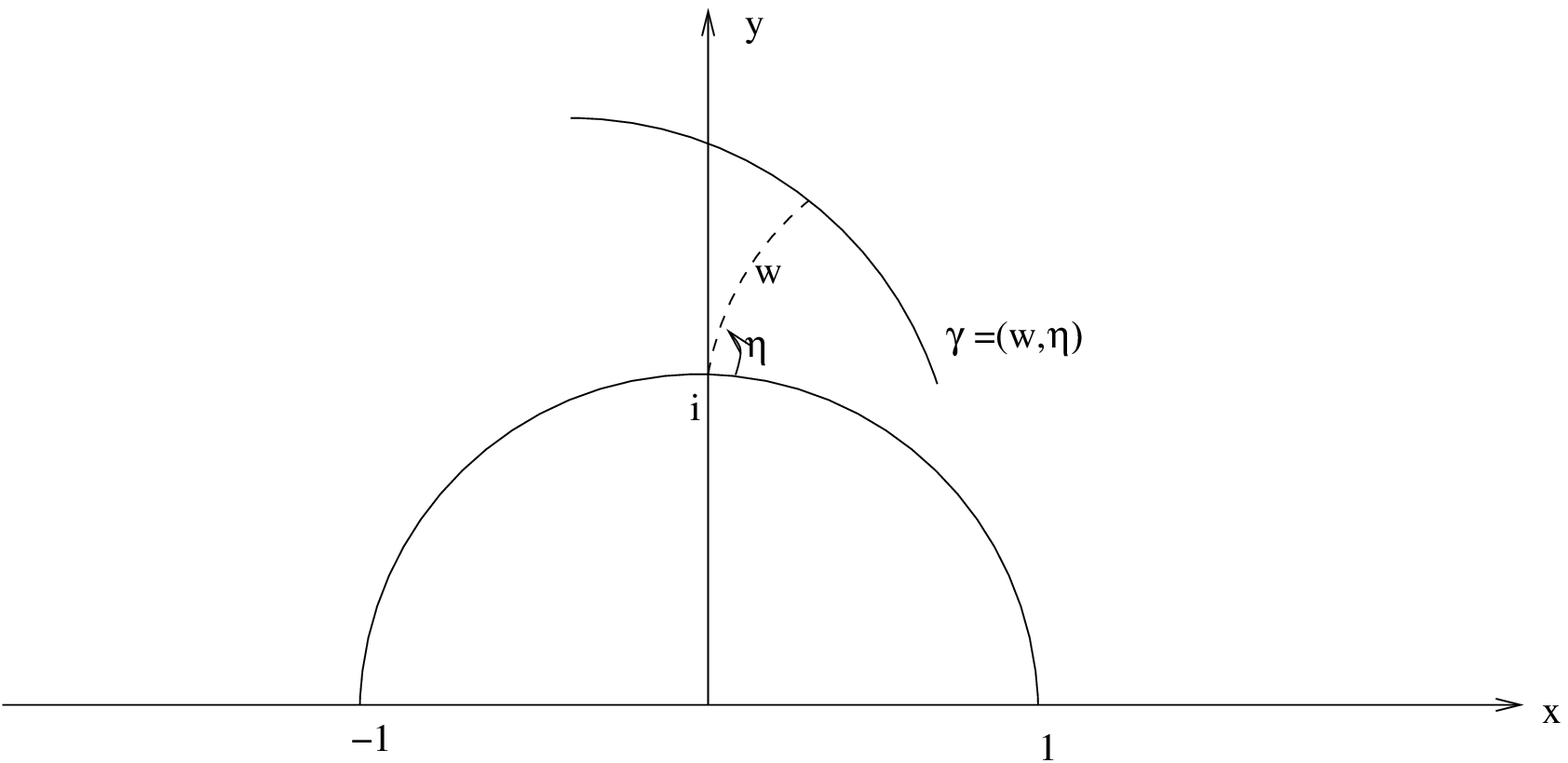}
\end{center}

The expression of $\mu$ in these parameters can be computed from the above local expression by using hyperbolic trigonometry.

Let $(l,\alpha)$ be the parameters with respect to the imaginary axis where the origin is the point $i$. For a geodesic having both parameters $(l,\alpha)$ and $(w,\eta)$, the relation between its two parametrizations is given by the following formulas:
\begin{eqnarray*}
&&\tanh l=\frac{\tanh w}{\cos\eta},\\
&&\cos\alpha=\cosh w\sin\eta.
\end{eqnarray*}

By changing variables, we obtain the formula for $\mu$ in terms of the parameters $(u,\eta)$:
\begin{equation}\label{po}
{\rm d}\mu=\frac{1}{2}\cosh w\,{\rm d}w{\rm d}\eta.
\end{equation}
By considering the rotations of $\mathbb{H}$ with the center $i$, Formula (\ref{po}) is well defined for all geodesics in $\mathbb{H}$. 

\section{The hyperbolic Ambartzumian-Pleijel identity}

In this section, we give proofs of Theorem \ref{1} and Theorem \ref{2}.

\subsection{Proof of Theorem \ref{1}}
\begin{proof}
The idea of proof is from \cite{cabo}. Recall that $D$ is a convex compact domain in $\mathbb{H}$ whose boundary is a geodesic polygon and $a_1,\dots,a_n$ are the edges of $\partial D$. Assume that the edges are labeled counter-clockwise. Recall that $\mathcal{G}_D$ is the subset of $\mathcal{G}_{\mathbb{H}}$ consisting of the geodesics intersecting $D$. We parametrize the geodesic in $\mathcal{G}_D$ by a pair of distinct boundary points. Then there is a bijection between $\mathcal{G}_D$ and $(\bigcup\limits_{j>k}a_j\times a_k)\setminus Z$ where $Z$ comes from the multiplicities of the diagonals of $\partial D$ in $\bigcup\limits_{j>k}a_j\times a_k$. As the Liouville measure has no atom, we have $\mu(Z)=0$.

Let $f$ be a $C^1$ function from $\mathbb{R}$ to $\mathbb{R}$. Recall that $\rho_D$ is the chord length with respect to $D$. We have the following equality:
\begin{equation}\label{int}
 \int_{\mathcal{G}_D} f(\rho)\,{\rm d}\mu=\sum_{j>k}\left(\int_{a_j\times a_k} f(\rho)\,{\rm d}\mu\right).
\end{equation}
Consider the orientation on the boundary $\partial D$ such that $D$ is on its left. Then the chord of $\gamma\in \mathcal{G}_D$ with respect to $D$ is on the left for each $a_j$. Consider a pair of edges $(a_j,a_k)$ with $j>k$. Choose and fix parametrizations of the geodesics associated to $a_j$ and $a_k$. Let $l_j$ and $l_k$ denote the parameters associated to $a_j$ and $a_k$ respectively. 
\begin{center}
 \includegraphics[scale=0.6]{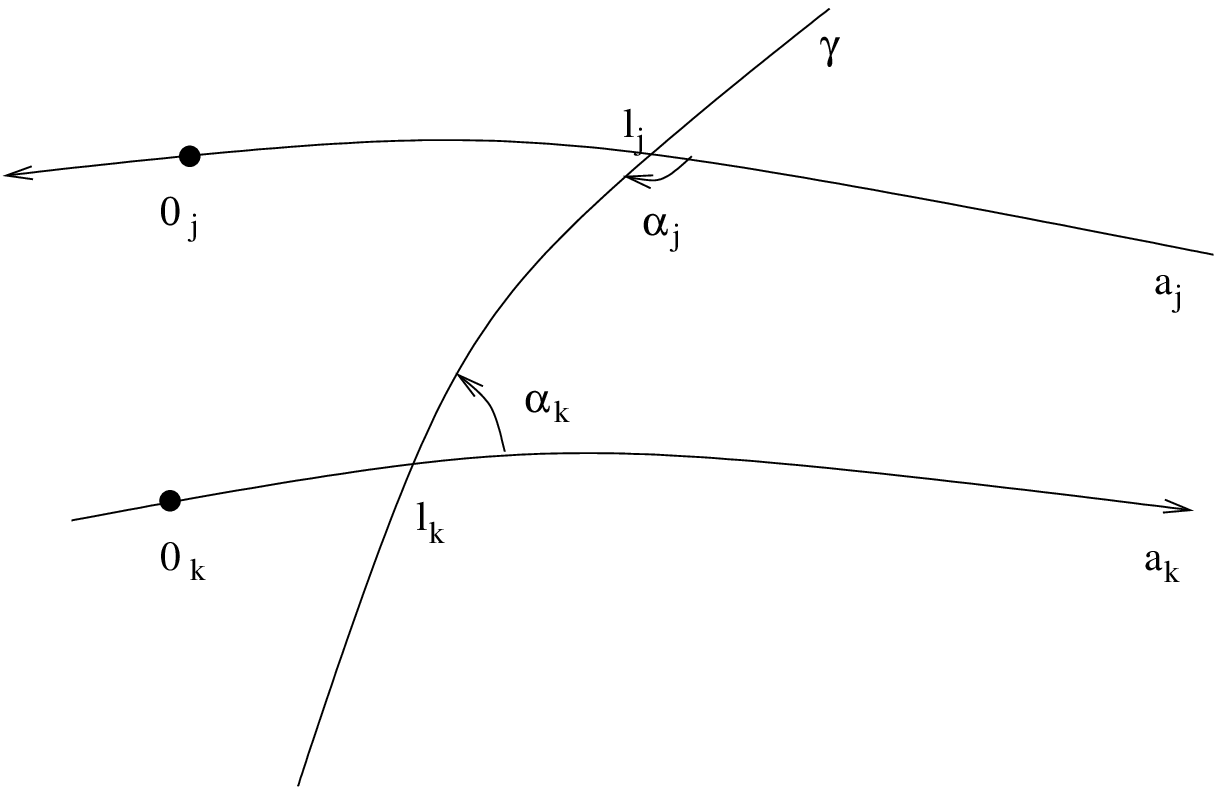}
\end{center}

The local expression of $\mu$ with respect to them is given by: 
\begin{equation}\label{p0}
{\rm d}\mu=\frac{\sin\alpha_j\sin\alpha_k}{2\sinh\rho} \,{\rm d}l_j\wedge {\rm d}l_k.
\end{equation}

By hyperbolic trigonometry, we have:
\begin{equation}\label{p1}
 \frac{\partial\rho}{\partial l_j}=\cos\alpha_j,
\end{equation}
and
\begin{equation}\label{p2}
 \frac{\partial\rho}{\partial l_k}=-\cos\alpha_k.
\end{equation}

Now, consider the following $1$-form:
\begin{displaymath}
\omega_{jk}=-\frac{\cos\alpha_j}{4} \,{\rm d}l_j-\frac{\cos\alpha_k}{4} \,{\rm d}l_k.
\end{displaymath}
which is well-defined on a rectangular domain in $\mathbb{R}^2$ where each point $(l_j,l_k)$ represents the geodesic intersecting $a_j$ (respectively $a_k$) at $l_j$ (respectively $l_k$).

By changing variables, we have that:
\begin{eqnarray*}
{\rm d}\omega_{jk}&=&\frac{\sin\alpha_j}{4} \,{\rm d}\alpha_j\wedge {\rm d}l_j+\frac{\sin\alpha_k}{4} \,{\rm d}\alpha_k\wedge {\rm d}l_k\\
&=&-\frac{\sin\alpha_j\sin\alpha_k}{4\sinh\rho} \,{\rm d}l_k\wedge {\rm d}l_j+\frac{\sin\alpha_j\sin\alpha_k}{4\sinh\rho} \,{\rm d}l_j\wedge {\rm d}l_k\\
&=&\frac{\sin\alpha_j\sin\alpha_k}{2\sinh\rho} \,{\rm d}l_j\wedge {\rm d}l_k\\
&=&{\rm d}\mu.
\end{eqnarray*}

We will compute the right hand side of (\ref{int}) term by term. For a pair of sides $(a_j,a_k)$ with $j>k$, by Stokes' formula one has:

\begin{equation}\label{stokes}
 \int_{\partial (a_j\times a_k)} f(\rho)\omega_{jk}=\int_{a_j\times a_k} f'(\rho)\,{\rm d}\rho\wedge\omega_{jk}+\int_{a_j\times a_k} f(\rho)\,{\rm d}\omega_{jk}.
\end{equation}

In the first term on the right hand side:
\begin{eqnarray*}
 {\rm d}\rho\wedge\omega_{jk}&=&(\cos\alpha_j \,{\rm d}l_j-\cos\alpha_k \,{\rm d}l_k)\wedge(-\frac{\cos\alpha_j}{4} \,{\rm d}l_j-\frac{\cos\alpha_k}{4} \,{\rm d}l_k)\\
&=&-\frac{\cos\alpha_j\cos\alpha_k}{2}\,{\rm d}l_j\wedge {\rm d}l_k.
\end{eqnarray*}

Comparing it with the Formula (\ref{p0}), we find the following relation:

\begin{eqnarray*}
 {\rm d}\rho\wedge\omega_{jk}&=&-\frac{\cos\alpha_j\cos\alpha_k}{2}\,{\rm d}l_j\wedge {\rm d}l_k\\
&=&-\frac{\cos\alpha_j\cos\alpha_k}{2}\frac{2\sinh\rho}{\sin\alpha_j\sin\alpha_k}\,{\rm d}\mu\\
&=&-\cot\alpha_j\cot\alpha_k\sinh\rho \,{\rm d}\mu.
\end{eqnarray*}

So this first term on the right hand side of (\ref{stokes}) becomes:
\begin{equation}\label{rightfirst}
 \int_{a_j\times a_k} f'(\rho)\,{\rm d}\rho\wedge\omega_{jk}=-\int_{a_j\times a_k} f'(\rho)\cot\alpha_j\cot\alpha_k\sinh\rho \,{\rm d}\mu.
\end{equation}

Now we turn to the left hand side of (\ref{stokes}). We need to discuss two cases depending on the relative positions between $a_j$ and $a_k$:

\textbf{(I)} The edges $a_j$ and $a_k$ are not adjacent (or equivalently $|j-k|\neq 1$ mod $n$).

We denote by $A_j$ and $B_j$ (respectively $A_k$ and $B_k$) the starting and end points of $a_j$ (respectively $a_k$). Then the left hand side of (\ref{stokes}) can be computed as following:

\begin{eqnarray}\label{cancel}
\int_{\partial (a_j\times a_k)} f(\rho)\omega_{jk}
&=&-\int^{(A_j,B_k)}_{(A_j,A_k)}f(\rho)\omega_{jk}+\int^{(B_j,B_k)}_{(B_j,A_k)}f(\rho)\omega_{jk}\nonumber\\
&&+\int^{(B_j,A_k)}_{(A_j,A_k)}f(\rho)\omega_{jk}-\int^{(B_j,B_k)}_{(A_j,B_k)}f(\rho)\omega_{jk}\nonumber\\
&=&\int^{(A_j,B_k)}_{(A_j,A_k)}f(\rho)\frac{\cos\alpha_k}{4}\,{\rm d}l_k-\int^{(B_j,B_k)}_{(B_j,A_k)}f(\rho)\frac{\cos\alpha_k}{4} \,{\rm d}l_k\nonumber\\
&&-\int^{(B_j,A_k)}_{(A_j,A_k)}f(\rho)\frac{\cos\alpha_j}{4}\,{\rm d}l_j+\int^{(A_j,B_k)}_{(A_j,A_k)}f(\rho)\frac{\cos\alpha_j}{4}\,{\rm d}l_j\nonumber\\
&=&\frac{1}{4}(-\int^{\rho_1}_{\rho_3}+\int^{\rho_2}_{\rho_4}-\int^{\rho_4}_{\rho_3}+\int^{\rho_2}_{\rho_1}) f(\rho)\,{\rm d}\rho.
\end{eqnarray}
where ${\rho_1}$ is the length of the diagonal $(A_j,B_k)$, ${\rho_2}$ is the length of the diagonal $(B_j,B_k)$, ${\rho_3}$ is the length of the diagonal $(A_j,A_k)$ and ${\rho_4}$ is the length of the diagonal $(B_j,A_k)$. The last equality comes from the change of variable using (\ref{p1}) and (\ref{p2}).

\textbf{(II)} The edges $a_j$ and $a_k$ are adjacent (or equivalently $|j-k|=1$ mod $n$).

Without loss of generality, we can assume that $B_j=A_k$. In the same way as in \textbf{(I)}, we get the following equality:

\begin{eqnarray*}
\int_{\partial (a_j\times a_k)} f(\rho)\omega =\frac{1}{4}(-\int^{\rho_1}_{\rho_3}+\int^{\rho_2}_{\rho_4}-\int^{\rho_4}_{\rho_3}+\int^{\rho_2}_{\rho_1}) f(\rho)\,{\rm d}\rho.
\end{eqnarray*}
Moreover we have the following relations:
\begin{eqnarray*}
&&\rho_2=|a_k|,\\
&&\rho_3=|a_j|,\\
&&\rho_4=0.
\end{eqnarray*}
So in this case we obtain the following formula:
\begin{equation}\label{+1}
 \int_{\partial (a_j\times a_k)} f(\rho)\omega_{jk}=\frac{1}{4}(-\int^{\rho_1}_{|a_j|}+\int^{|a_k|}_{0}+\int^{|a_j|}_{0}+\int^{|a_k|}_{\rho_1}) f(\rho)\,{\rm d}\rho.
\end{equation}

The last step is to sum up the formulas (\ref{stokes}) for each $(j,k)$. Let us first compute the left hand side:
\begin{equation*}
 \sum_{j>k}\int_{\partial (a_j\times a_k)} f(\rho)\omega =\sum_{j>k}\frac{1}{4}(-\int^{\rho_1}_{\rho_3}+\int^{\rho_2}_{\rho_4}-\int^{\rho_4}_{\rho_3}+\int^{\rho_2}_{\rho_1}) f(\rho)\,{\rm d}\rho
\end{equation*}

For any $(j,k)$ in Case \textbf{(I)}, the first term $$-\int^{\rho_1}_{\rho_3}f(\rho)\,{\rm d}\rho,$$ on the right hand side of the associated Formula (\ref{cancel}) will also appear on the right hand side of Formula (\ref{cancel}) associated to $(j-1,k)$, but with a different sign. The same happens for the other three terms for $(j,k)$. For any $(j,k)$ in Case \textbf{(II)}, this happens for the first and the last term in Formula (\ref{+1}). 

After summing up Formula (\ref{cancel}) and Formula (\ref{+1}) for all pairs $(j,k)$, the terms different by a sign will cancel each other and the rest is the following:
\begin{eqnarray}\label{left}
\sum_{j>k} \int_{\partial (a_j\times a_k)} f(\rho)\omega_{jk}=\frac{1}{2}\sum^n_{i=1}\int^{|a_i|}_0f(x)\,{\rm d}x.	
\end{eqnarray}

For the first term on the right hand side in the sum, by Formula (\ref{rightfirst}), we have:
\begin{eqnarray*}
 \sum_{j>k}\int_{a_j\times a_k} f'(\rho)\,{\rm d}\rho\wedge\omega_{jk}&=&\sum_{j>k}-\int_{a_j\times a_k} f'(\rho)\cot\alpha_j\cot\alpha_k\sinh\rho \,{\rm d}\mu\\
&=&-\int_{\mathcal{G}_D} f'(\rho)\cot\alpha_j\cot\alpha_k\sinh\rho \,{\rm d}\mu.
\end{eqnarray*}

By moving it to the left hand side, we finally get the formula in Theorem \ref{1}.
\end{proof}

\subsection{Proof of Theorem \ref{2}}
\begin{proof}
Given a convex compact domain $D$ with $C^1$ boundary in $\mathbb{H}$, we can choose $3$ points $b_1,b_2$ and $b_3$ in $\partial D$ and get a triangle $D_3$ inscribed into $D$. Then by Theorem \ref{1}, we have the Ambartzumian-Pleijel identity for $D_3$:
\begin{equation*}
 \int_{\mathcal{G}_{D_3}}(f\circ\rho_3)\,{\rm d}\mu=\int_{\mathcal{G}_{D_3}}(f'\circ\rho_3)\sinh\rho_3\cot\alpha_1\cot\alpha_2 \,{\rm d}\mu+\frac{1}{2}\sum^3_{i=1}\int^{|a_i|}_0f(x)\,{\rm d}x.
\end{equation*}
Each pair of the adjacent vertices $(b_j,b_{j+1})$ of $D_3$ separates $\partial D$ into two parts and exactly one of which containing no vertices of $D_3$. Denote this arc by $\gamma_j$. Then we consider a new set of points in $\partial D$ consisting all $b_j$ and the mid-point of $\gamma_j$ for all $j$. The corresponding inscribed polygon denoted by $D_{6}$ will gives us another Ambartzumian-Pleijel identity:
\begin{equation*}
 \int_{\mathcal{G}_{D_{6}}}(f\circ\rho_{6})\,{\rm d}\mu=\int_{\mathcal{G}_{D_{6}}}(f'\circ\rho_{6})\sinh\rho_{6}\cot\alpha_1\cot\alpha_2 \,{\rm d}\mu+\frac{1}{2}\sum^{6}_{i=1}\int^{|a_i|}_0f(x)\,{\rm d}x.
\end{equation*}

Repeating the above construction for $D_{6}$ and so on, we can get a sequence of polygons $D_{3n}$. As $D$ is compact convex with $C^1$ boundary and $f$ is $C^1$, the function $(f\circ\rho_{3n})\chi_{\mathcal{G}_{D_{3n}}}$ defined on $\mathcal{G}_\mathbb{H}$ uniformly converges to $(f\circ\rho)\chi_{\mathcal{G}_D}$ when $n$ does to $\infty$, where $\rho$ is the chord length function for $D$. Also the maximum of lengths of boundary segments for each $D_{3n}$ will go to $0$ when $n$ goes to $\infty$. By passing to the limit, we obtain the required formula.
\end{proof}

\section{Applications}
\subsection{The Liouville measure of $\mathcal{G}_D$}
The proof of Corollary \ref{c1} is direct.
\begin{proof}
We choose $f$ to be the constant map: $f(x)=1$. Since the derivative of $f$ is identically zero, the first term of the right hand side is $0$. The corollary follows.
\end{proof}
\subsection{Isoperimetric inequality}
The idea of the proof comes from that for the Euclidean case using the Pleijel identity given by Ambartzumian in \cite{amba}.

\begin{proof}
In the proof of Theorem \ref{1}, we obtained the following two formulas for a polygon domain $D$:
\begin{equation*}
\int_{\mathcal{G}_D}(f\circ\rho)\,{\rm d}\mu=\frac{1}{2}\int_{\mathcal{G}_D}(f'\circ\rho)\cos\alpha_1\cos\alpha_2\,{\rm d}l_1{\rm d}l_2+\frac{1}{2}\sum^n_{i=1}\int^{|a_i|}_0f(x)\,{\rm d}x,
\end{equation*}
\begin{equation*}
\int_{\mathcal{G}_D}(f\circ\rho)\,{\rm d}\mu=\frac{1}{2}\int_{\mathcal{G}_D}(f\circ\rho)\frac{\sin\alpha_1\sin\alpha_2}{\sinh\rho}\,{\rm d}l_1{\rm d}l_2.
\end{equation*}

By the same argument as in the proof of Theorem \ref{2}, we have the following two equalities for those $D$ with $C^1$ boundary:
\begin{equation}\label{ii1}
\int_{\mathcal{G}_D}(f\circ\rho)\,{\rm d}\mu=\frac{1}{2}\int_{\mathcal{G}_D}(f'\circ\rho)\cos\alpha_1\cos\alpha_2\,{\rm d}l_1{\rm d}l_2+\frac{1}{2}f(0)L(\partial D),
\end{equation}
\begin{equation}\label{ii2}
\int_{\mathcal{G}_D}(f\circ\rho)\,{\rm d}\mu=\frac{1}{2}\int_{\mathcal{G}_D}(f\circ\rho)\frac{\sin\alpha_1\sin\alpha_2}{\sinh\rho}\,{\rm d}l_1{\rm d}l_2.
\end{equation}

In (\ref{ii1}) we take $f$ to be $f(x)=x$ and in (\ref{ii2}) we take $f$ to be $f(x)=\sinh x$, which yield the following two equalities:
 
\begin{equation}\label{iii1}
\int_{\mathcal{G}_D}\rho\,{\rm d}\mu=\frac{1}{2}\int_{\mathcal{G}_D}\cos\alpha_1\cos\alpha_2\,{\rm d}l_1{\rm d}l_2,
\end{equation}
\begin{equation}\label{iii2}
\int_{\mathcal{G}_D}(\sinh\circ\rho)\,{\rm d}\mu=\frac{1}{2}\int_{\mathcal{G}_D}\sin\alpha_1\sin\alpha_2\,{\rm d}l_1{\rm d}l_2.
\end{equation}

Adding the two equalities above, on the right hand side we get the following:
\begin{eqnarray*}
&&\frac{1}{2}\int_{\mathcal{G}_D}(\cos\alpha_1\cos\alpha_2+\sin\alpha_1\sin\alpha_2)\,{\rm d}l_1{\rm d}l_2\\
&=&\frac{1}{2}\int_{\mathcal{G}_D}\cos(\alpha_1-\alpha_2)\,{\rm d}l_1{\rm d}l_2\\
&=&\frac{1}{4}\int_{\partial D}\int_{\partial D}(1-2\sin^2(\frac{\alpha_1-\alpha_2}{2}))\,{\rm d}l_1{\rm d}l_2\\
&=&\frac{1}{4}(L(\partial D))^2-\frac{1}{2}\int_{\partial D}\int_{\partial D}\sin^2(\frac{\alpha_1-\alpha_2}{2})\,{\rm d}l_1{\rm d}l_2,
\end{eqnarray*}
whilst on the left hand side, we have the following:
\begin{equation}\label{AA}
\int_{\mathcal{G}_D}(\rho+\sinh(\rho))\,{\rm d}\mu=\int_{\mathcal{G}_D}2\rho\,{\rm d}\mu+\int_{\mathcal{G}_D}(\sinh(\rho)-\rho)\,{\rm d}\mu.
\end{equation}

The first term on the right hand side of Formula (\ref{AA}) is the volume of the unit tangent bundle over $D$. So the integral equals $\pi A(D)$. To compute the second term on the right hand side of (\ref{AA}), we need to use the hyperbolic volume form that we found in the former sections: denote by $P_1$ and $P_2$ the points in $D$, then we have:
\begin{eqnarray}
 (A(D))^2&=&\int_D\int_D{\rm dVol}(P_2)\, {\rm dVol}(P_1)\\
&=&\int_D(\int_D\sinh r_1(P_2)\,{\rm d}r_1(P_2){\rm d}\theta_1(P_2)){\rm dVol}(P_1),
\end{eqnarray}
where $(r_1,\theta_1)$ are the polar coordinates of $P_2$ with respect to $P_1$.
\begin{center}
 \includegraphics[scale=0.4]{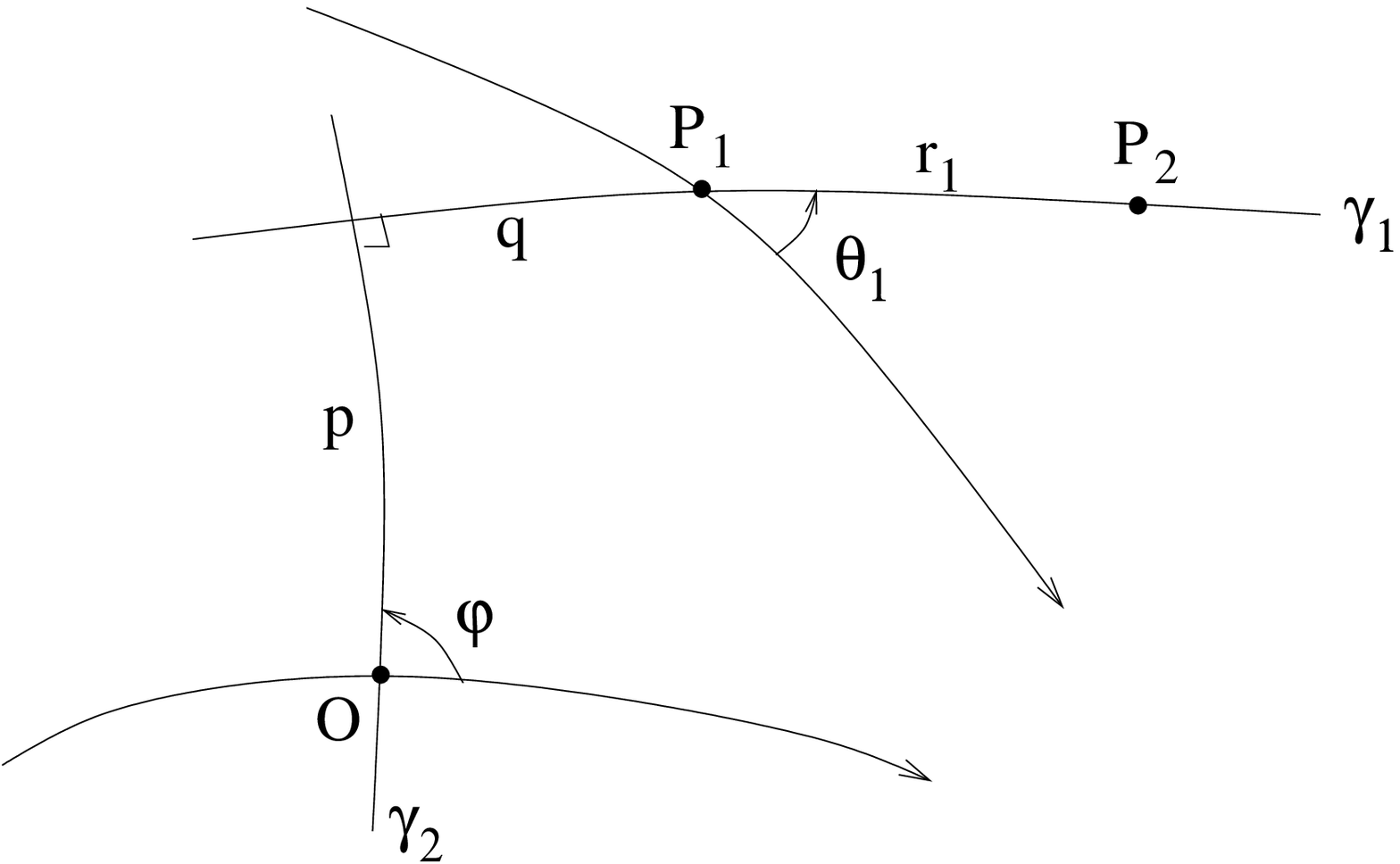}
\end{center}

Consider the geodesic $\gamma_1$ passing through the origin $i\in\mathbb{H}$ and orthogonal to the geodesic $\gamma_2$ passing through $P_1$ and $P_2$. It has an angle $\phi$ to a fixed geodesic ray based on $O$. We can change the parameter from $(r_1,\theta_1)$ to $(r_1,\phi)$ and the resulting formula is: 
\begin{equation}
(A(D))^2=\int_D\int_D\sinh r_1\frac{\cosh p}{\cosh q}\,{\rm d}r_1{\rm d}\phi{\rm dVol}(P_1).
\end{equation}
where $p$ is the distance from he origin to $\gamma_1$ and $q$ is the distance from the foot of orthogonal projection of the origin on $\gamma_2$ to $P_1$. Consider the rectangular coordinates of $P_1$ with respect to $\gamma_1$. Then we have the 
\begin{eqnarray*}
 \sinh r_1\frac{\cosh p}{\cosh q}\,{\rm d}r_1{\rm d}\phi{\rm dVol}(P_1)&=&\sinh r_1\frac{\cosh p}{\cosh q}\,{\rm d}r_1{\rm d}\phi\cosh q\,{\rm d}p{\rm d}q\\
&=&\sinh r_1\cosh p\,{\rm d}r_1{\rm d}\phi{\rm d}p{\rm d}q\\
&=&2\sinh r_1\,{\rm d}r_1{\rm d}q{\rm d}\mu.
\end{eqnarray*}
where the last equality comes from the formula (\ref{po}).

Consider a parametrization on geodesic $\gamma_2$. Let $p$ and $p'$ denote the positions of $P_1$ and $P_2$ respectively. Then we have $r=|p-p'|$. The above formula can be rewritten as following:
\begin{eqnarray}
 (A(D))^2&=&\int_{\mathcal{G}_D}\int_0^\rho\int_0^\rho2\sinh r_1\,{\rm d}q'{\rm d}q{\rm d}\mu\\
&=&\int_{\mathcal{G}_D}4(\sinh\rho-\rho)\,{\rm d}\mu.
\end{eqnarray}

By all the computations above, the sum of (\ref{iii1}) and (\ref{iii2}) can be rewritten as following:
\begin{equation}\label{ii}
 \frac{1}{4}(L(\partial D))^2-\frac{1}{2}\int_{\partial D}\int_{\partial D}\sin^2(\frac{\alpha_1-\alpha_2}{2})\,{\rm d}l_1{\rm d}l_2=\pi A(D)+\frac{1}{4}(A(D))^2.
\end{equation}
which implies the isoperimetric inequality. Also the formula (\ref{ii}) shows that the equality holds if and only if $\alpha_1=\alpha_2$ for all $\gamma\in\mathcal{G}_D$ which implies $\partial D$ is a circle.
\end{proof}
\begin{remark}
 The computation for $(A(D))^2$ is due to Santal\'o in \cite{santalo}.
\end{remark}

\subsection{Chord length distribution in an ideal triangle}
In this section we give a computation of the distribution of the chord length in an ideal triangle with respect to the Liouville measure. We remark that this result has been previously obtained by Bridgeman and Dumas in \cite{martin}. Here we give a different approach.

\begin{proof}
It is well known that all ideal triangles in $\mathbb{H}^2$ are isometric. We can assume that the vertices of the ideal triangle are $0$,$1$ and $\infty$. These points separate the boundary of hyperbolic plane into three intervals: $I_1=]\infty,0[$, $I_2=]0,1[$ and $I_3=]1,\infty[$. The set of the geodesics crossing the ideal triangle is $\bigcup\limits_{j<k}I_j\times I_k$. Because of the symmetry of the ideal triangle, we only need to consider $I_1\times I_3$ and we denote it by $\mathcal{G}_0$.

We parametrize $i\mathbb{R}$ such that the point $z_1$ with the coordinate $l_1$ is the point $ie^{l_1}$. In the same way, we parametrize $1+i\mathbb{R}$ such that the point $z_3$ with the coordinate $l_3$ is $1+ie^{l_3}$. Then a geodesic $\gamma\in\mathcal{G}_0$ can be parametrized by $(l_1,l_3)$. by the formula that we obtained in section 2, the Liouville measure can be expressed locally as following:
\begin{equation*}
d\mu=\frac{\sin\alpha_1\sin\alpha_3}{2\sinh\rho}\,{\rm d}l_1{\rm d}l_3,
\end{equation*}
where $\alpha_1$ (resp. $\alpha_2$) is the angle between $\gamma\in \mathcal{G}_0$ and $i\mathbb{R}$ (reps. $1+i\mathbb{R}$).

The chord length $\rho$ of $\gamma$ can be computed as following:
\begin{equation*}
\rho=\log\frac{|z_1-\bar {z_3}|+|z_1-z_3|}{|z_1-\bar {z_3}|-|z_1-z_3|}.
\end{equation*}
Equivalently we have:
\begin{eqnarray*}
\sinh\rho&=&\frac{|z_1-\bar {z_3}||z_1-z_3|}{2e^{l_1}e^{l_3}},\\
\cosh\rho&=&\frac{1+e^{2l_1}+e^{2l_3}}{2e^{l_1}e^{l_3}}.
\end{eqnarray*}

The center and the radius of $\gamma$ are the following:
\begin{equation*}
c=\frac{1+e^{2l_3}-e^{2l_1}}{2},
\end{equation*}
\begin{equation*}
r=\sinh\rho {e^{l_1}e^{l_3}}.
\end{equation*}

Then we can write the $\sin\alpha_1$ and $\sin\alpha_3$ as functions of $\rho,{l_1},{l_3}$:
\begin{equation*}
\sin\alpha_1=\frac{e^{l_1}}{r}=\frac{1}{\sinh\rho e^{l_3}},
\end{equation*}
\begin{equation*}
\sin\alpha_3=\frac{e^{l_3}}{r}=\frac{1}{\sinh\rho e^{l_1}}.
\end{equation*}

Then $d\mu$ is rewritten as following:
\begin{equation*}
d\mu=\frac{1}{\sinh^3\rho e^{l_1}e^{l_3}}\,{\rm d}{l_1}{\rm d}{l_3}.
\end{equation*}

By considering the equation:
\begin{equation*}
e^{2{l_3}}-2e^{l_1}e^{l_3}\cosh\rho+1+e^{2{l_1}}=0,
\end{equation*}
the parameter $k$ can be expressed by a function of $\rho$ and ${l_1}$:
\begin{equation*}
e^{l_3}=e^{l_1}\cosh\rho\pm\sqrt{e^{2{l_1}}\sinh^2\rho-1}.
\end{equation*}

These two solutions correspond to two different cases. Fix the point $z_1$ and move the point $z_3$ from $1$ to $\infty$ along the geodesic $1+i\mathbb{R}$. The chord length $\rho$ decreases from $\infty$ to a minimal value then increases back to $\infty$ where the minimal value is the distance between $z_1$ and the geodesic $1+i\mathbb{R}$. We denote this distance by $d(z_1)$. This means that in generic case for a fixed  $z_1$ and a fixed $\rho$, there are two points $z_3$ and $z_3'$ in $1+i\mathbb{R}$ satisfying that their distances to $z_1$ are both $\rho$. These two points correspond to the two solutions above respectively.

For a fixed $l_1$, to compute $l_3$ realizing $d(z_1)$, we consider the geodesic passing through $z_1$ and perpendicular to $1+i\mathbb{R}$. The center of this geodesic is $1$, and the foot of the perpendicular geodesic is $1+i\sqrt{1+e^{2l_1}}$ which implies that $e^{2{l_3}}=1+e^{2l_1}$. Then we can separate $\mathbb{R}^2$ into two parts:
\begin{equation*}
U_1=\{(l_1,l_3)\in\mathbb{R}^2\mid e^{2l_3}\ge1+e^{2l_1}\},
\end{equation*}
and
\begin{equation*}
U_2=\{(l_1,l_3)\in\mathbb{R}^2\mid e^{2l_3}\le1+e^{2l_1}\},
\end{equation*}
such that in each part, there is at most one $l_3$ for each pair $(l_1,\rho)$.

Now we begin to find which solution corresponds to which subset.

When $\rho=d(z_1)$, we have the equality $e^{l_3}=e^{l_1}\cosh d(z_1)=\sqrt{1+e^{2l_1}}$. Then as $l_3$ increasing, we have the inequality:
\begin{equation*}
 e^{l_3}>e^{l_1}\cosh d(z_1),
\end{equation*}
and $\rho$ increases at the same time. This tells us that the solution for the set $U_1$ is:
\begin{equation*}
e^{l_3}=e^{l_1}\cosh\rho+\sqrt{e^{2l_1}\sinh^2\rho-1},
\end{equation*}
Then the solution for $U_2$ is:
\begin{equation*}
e^{l_3}=e^{l_1}\cosh\rho-\sqrt{e^{2l_1}\sinh^2\rho-1}.
\end{equation*}

The integral splits into two parts:
\begin{equation*}
\int_{\mathcal{G}_0}f(\rho)\,{\rm d}\mu=\int_{-\infty}^{+\infty}\int_{\sqrt{1+e^{2l_1}}}^{+\infty}\frac{f(\rho)}{2e^{l_1}e^{l_3}\sinh^3\rho}\,{\rm d}l_3{\rm d}l_1+
\int_{-\infty}^{+\infty}\int^{\sqrt{1+e^{2l_1}}}_{-\infty}\frac{f(\rho)}{2e^{l_1}e^{l_3}\sinh^3\rho}\,{\rm d}l_3{\rm d}l_1
\end{equation*}
where $f:\mathbb{R}\rightarrow\mathbb{R}$ is $C^1$ and with compact support. We use $I$ and $II$ to denote the first term and the second term of the right hand side.

Recall that in $I$ we have
\begin{equation*}
e^{l_3}=e^{l_1}\cosh\rho+\sqrt{e^{2l_1}\sinh^2\rho-1}.
\end{equation*}
from which we compute the determinant of Jacobi as following:
\begin{eqnarray*}
|\frac{\partial l_3}{\partial\rho}|&=&\frac{1}{e^{l_3}}[e^{l_1}\sinh\rho+\frac{e^{2l_1}\sinh\rho\cosh\rho}{\sqrt{e^{2l_1}\sinh^2\rho-1}}]\\
&=&\frac{e^{l_1}\sinh\rho}{e^{l_3}\sqrt{e^{2l_1}\sinh^2\rho-1}}(\sqrt{e^{2l_1}\sinh^2\rho-1}+e^{l_1}\cosh\rho)\\
&=&\frac{e^{l_1}\sinh\rho}{e^{l_3}\sqrt{e^{2l_1}\sinh^2\rho-1}}e^{l_3}\\
&=&\frac{e^{l_1}\sinh\rho}{\sqrt{e^{2l_1}\sinh^2\rho-1}}.
\end{eqnarray*}

Then we have:
\begin{eqnarray*}
I&=&\int_{-\infty}^{+\infty}\int_{\sqrt{1+e^{2l_1}}}^{+\infty}\frac{f(\rho)}{2e^{l_1}e^{l_3}\sinh^3\rho}\,{\rm d}l_3{\rm d}l_1\\
&=&\int_{-\infty}^{+\infty}\int_{\sinh^{-1}(\frac{1}{e^{l_1}})}^{+\infty}\frac{f(\rho)}{2e^{l_1}e^{{l_3}}\sinh^3\rho}\frac{e^{l_1}\sinh\rho}{\sqrt{e^{2l_1}\sinh^2\rho-1}}\,{\rm d}\rho {\rm d}l_1\\
&=&\int_{-\infty}^{+\infty}\int_{\sinh^{-1}(\frac{1}{e^{l_1}})}^{+\infty}\frac{f(\rho)}{{2(\sqrt{e^{2l_1}\sinh^2\rho-1}+e^{l_1}\cosh\rho)}\sinh^2\rho\sqrt{e^{2l_1}\sinh^2\rho-1}}\,{\rm d}\rho {\rm d}l_1\\
&=&\int_{0}^{+\infty}\int_{-\ln\sinh\rho}^{+\infty}\frac{f(\rho)}{{2(\sqrt{e^{2l_1}\sinh^2\rho-1}+e^{l_1}\cosh\rho)}\sinh^2\rho\sqrt{e^{2l_1}\sinh^2\rho-1}}\,{\rm d}l_1{\rm d}\rho.
\end{eqnarray*}

Let
\begin{equation*}
 v=(\sinh\rho+\cosh\rho)e^{l_1}(\sinh\rho e^{l_1}+\sqrt{\sinh^2\rho e^{2l_1}-1}).
\end{equation*}

We change variables from $(l_1,\rho)$ to $(v,\rho)$ and the Jacobi is:
\begin{equation*}
 |J|=|\frac{\partial l_1}{\partial v}|=[(\sinh\rho+\cosh\rho)e^{2l_1}(\sinh\rho e^{l_1}+\sqrt{\sinh^2\rho e^{2l_1}-1})^2\frac{\sinh\rho}{\sqrt{\sinh^2\rho e^{2l_1}-1}}]^{-1}.
\end{equation*}

So the integral $I$ can be written as following:
\begin{eqnarray*}
 I&=&\int_{0}^{+\infty}\int_{\frac{\cosh\rho+\sinh\rho}{\sinh\rho}}^{+\infty}\frac{f(\rho)}{{2(\sqrt{e^{2l_1}\sinh^2\rho-1}+e^{l_1}\cosh\rho)}\sinh^2\rho\sqrt{e^{2l_1}\sinh^2\rho-1}}\times\\
 &&\times\frac{\sqrt{e^{2l_1}\sinh^2\rho-1}}{(\cosh\rho+\sinh\rho)e^{l_1}(\sinh\rho e^{l_1}+\sqrt{\sinh^2\rho e^{2l_1}-1})^2}\,{\rm d}v{\rm d}\rho\\
 &=&\int_{0}^{+\infty}\int_{\frac{\cosh\rho+\sinh\rho}{\sinh\rho}}^{+\infty}\frac{f(\rho)}{2\sinh^2\rho}\frac{1}{(\cosh\rho+\sinh\rho)e^{l_1}(\sinh\rho e^{l_1}+\sqrt{\sinh^2\rho e^{2l_1}-1})}\times\\
 &&\times\frac{1}{(\sqrt{e^{2l_1}\sinh^2\rho-1}+e^{l_1}\cosh\rho)(\sinh\rho e^{l_1}+\sqrt{\sinh^2\rho e^{2l_1}-1})}\,{\rm d}v{\rm d}\rho\\
 &=&\int_{0}^{+\infty}\int_{\frac{\cosh\rho+\sinh\rho}{\sinh\rho}}^{+\infty}\frac{f(\rho)}{2\sinh^2\rho v(v-1)}\,{\rm d}v{\rm d}\rho.
\end{eqnarray*}

We treat $II$ in a similar way. In this computation, we use:
\begin{equation*}
 v=(\sinh\rho+\cosh\rho)e^{l_1}(\sinh\rho e^{l_1}-\sqrt{\sinh^2\rho e^{2l_1}-1}),
\end{equation*}
to change variables.

Fix $\rho$ and compute the partial differential of $v$ with respect to $l_1$, we may find that it is always negative. We compute the limit of $v$ when $l$ goes to infinite as following:
\begin{eqnarray*}
 v&=&(\sinh\rho+\cosh\rho)e^{l_1}(\sinh\rho e^{l_1}-\sqrt{\sinh^2\rho e^{2l_1}-1})\\
  &=&(\sinh\rho+\cosh\rho)e^{l_1}(\sinh\rho e^{l_1}-\sinh\rho e^{l_1}(1-\frac{1}{2\sinh^2\rho e^{2l_1}}+\mathrm{o}(\frac{1}{\sinh^2\rho e^{2l_1}}))),{l_1\rightarrow+\infty}.
\end{eqnarray*}

This implies that:
\begin{equation*}
\lim_{l\rightarrow+\infty}v(\rho,l_1)=\frac{\sinh\rho+\cosh\rho}{2\sinh\rho},
\end{equation*}
for $\rho>0$ fixed.

Then the second part of the integral is:
\begin{equation*}
 II=\int_{0}^{+\infty}\int_{\frac{\cosh\rho+\sinh\rho}{2\sinh\rho}}^{\frac{\cosh\rho+\sinh\rho}{\sinh\rho}}\frac{f(\rho)}{2\sinh^2\rho v(v-1)}\,{\rm d}v{\rm d}\rho.
\end{equation*}

Putting $I$ and $II$ together:
\begin{eqnarray*}
 \int_{\mathcal{G}_0}f(\rho)d\mu&=&\int_{0}^{+\infty}\int_{\frac{\cosh\rho+\sinh\rho}{2\sinh\rho}}^{\infty}\frac{f(\rho)}{2\sinh^2\rho v(v-1)}\,{\rm d}v{\rm d}\rho\\
&=&\int_{0}^{+\infty}\frac{f(\rho)}{2\sinh^2\rho}(\int_{\frac{\cosh\rho+\sinh\rho}{2\sinh\rho}}^{\infty}\frac{1}{v(v-1)}\,{\rm d}v)\,{\rm d}\rho\\
&=&\int_{0}^{+\infty}\frac{f(\rho)}{2\sinh^2\rho}2\rho \,{\rm d}\rho\\
&=&\int_{0}^{+\infty}\frac{f(\rho)\rho}{\sinh^2\rho}\,{\rm d}\rho.
\end{eqnarray*}

This implies that:
\begin{displaymath}
  {\rm d}M_T=\frac{3\rho\, {\rm d}\rho}{\sinh^2\rho}.
 \end{displaymath}
which is same as described in \cite{martin}.
\end{proof}

\subsection{Chord length distribution in an ideal quadrilateral}
In this proof, the $\delta-$formalism is the main tool and this idea comes from \cite{gas} for the Euclidean version of the Pleijel identity.
\begin{proof}
Let $Q$ be an ideal quadrilateral. Let $\gamma_1,\gamma_2,\gamma_3$ and $\gamma_4$ denote its $4$ edges ordered counter-clockwise. The set of geodesics having at least one end at the vertices of $Q$ has the Liouville measure $0$. So we need only consider the following two types of geodesics intersecting $Q$: those intersecting two edges adjacent and those intersecting two opposite edges. The chord length distribution for the first type has been computed in the previous section and the main task in this section is to compute the chord length distribution for the second type.

Consider the geodesics $\gamma_1$ and $\gamma_3$. Let $\mathcal{G}_{13}$ denote the geodesic intersecting them. By using the $\delta-$formalism to the Pleijel identity, we can compute the Liouville measure of the following set:
\begin{equation*}
 \{\gamma\in\mathcal{G}_{13}:\rho(\gamma)\le\rho_0\}.
\end{equation*}

By considering the distance to $\gamma_3$, the geodesic $\gamma_1$ can be separated into three segments. The middle one consists of those points having distance to $\gamma_3$ smaller than $\rho_0$. The other two consists of all points having the distance to $\gamma_3$ strictly bigger than $\rho_0$. We choose one point in each of the latter two segments. Let $A_1$ and $B_1$ denote these two points then the geodesic segment $[A_1,B_1]$ contains the points in $\gamma_1$ having the distance to $\gamma_3$ smaller that $\rho_0$ and the distance of $A_1$ and $B_1$ to $\gamma_3$ are both bigger than $\rho_0$. By considering the distance to $\gamma_1$ and the same argument as for $\gamma_1$, we have two points $A_3$ and $B_3$ in $\gamma_3$ such that the geodesic segment $[A_3,B_3]$ contains the points in $\gamma_3$ having the distance to $\gamma_1$ smaller that $\rho_0$ and the distance of $A_3$ and $B_3$ to $\gamma_1$ are both bigger than $\rho_0$. 

We assume that the $A_1,B_1,A_3$ and $B_3$ are in the cyclic order in $\partial Q$. Let $\mathcal{G}_0$ be the set of geodesics intersecting $\gamma_1$ between $A_1$ and $B_1$ and intersecting $\gamma_3$ between $A_3$ and $B_3$. Then we have that:
\begin{equation*}
 \mu(\{\gamma\in\mathcal{G}_{13}:\rho(\gamma)\le\rho_0\})=\mu(\{\gamma\in\mathcal{G}_0:\rho(\gamma)\le\rho_0\}).
\end{equation*}
Then by using the Stokes' theorem and repeating the proof of Theorem \ref{1}, we have:
\begin{equation}\label{delta}
\int_{\mathcal{G}_0} f(\rho)\,{\rm d}\mu=\int_{\mathcal{G}_0} f'(\rho)\cot\alpha_1\cot\alpha_3\sinh\rho \,{\rm d}\mu+\frac{1}{4}(-\int^{\rho_1}_{\rho_3}+\int^{\rho_2}_{\rho_4}-\int^{\rho_4}_{\rho_3}+\int^{\rho_2}_{\rho_1}) f(\rho)\,{\rm d}\rho,
\end{equation}
where ${\rho_1}$ is the length of the diagonal $(A_3,B_1)$, ${\rho_2}$ is the length of the diagonal $(B_3,B_1)$, ${\rho_3}$ is the length of the diagonal $(A_3,A_1)$ and ${\rho_4}$ is the length of the diagonal $(B_3,A_1)$. By the assumptions for $A_1,B_1,A_3$ and $B_3$, we have that $\rho_1,\dots,\rho_4$ all bigger that $\rho_0$.

Instead of choosing $f$ to be $C^1$, we can formally choose $f$ to be the step function:
\begin{equation*}
 f(x) = \left\{ \begin{array}{ll}
1 & \textrm{if $x\le\rho_0$}\\
0 & \textrm{if $x>\rho_0$}
\end{array} \right..
\end{equation*}
Then its derivative $f'$ becomes a $\delta-$function:
\begin{equation*}
 f'(x)=\delta(\rho_0-x),
\end{equation*}
and the identity (\ref{delta}) becomes:
\begin{equation}
\mu(\{\gamma\in\mathcal{G}_{13}:\rho(\gamma)\le\rho_0\})=\int_{\mathcal{G}_D} \delta(\rho_0-\rho)\cot\alpha_1\cot\alpha_3\sinh\rho \,{\rm d}\mu.
\end{equation}

As $\gamma_1$ and $\gamma_3$ are disjoint, there is an unique geodesic $\gamma'$ orthogonal to both of them. By considering the mid point of the chord of $\gamma'$ as the origin and choose one oriented geodesic passing it, we can define the polar coordinates $(w,\eta)$ of $\mathcal{G}_{\mathbb{H}}$. Recall that under the polar coordinates, we have:
\begin{equation*}
 {\rm d}\mu=\frac{\cosh w}{2}\,{\rm d}w{\rm d}\eta.
\end{equation*}
Let $\gamma=(w,\eta)\in\mathcal{G}_0$. Let $\gamma_0$ be the geodesic passing the origin orthogonal to $\gamma$ at the point $z_0$. Let $z_1$ be the intersection point of $\gamma$ with $\gamma_1$ and $z_3$ be the intersection point of $\gamma$ with $\gamma_3$. Let $\rho_1$ denote the hyperbolic distance between $z_0$ and $z_1$ and $\rho_3$ denote the hyperbolic distance between $z_0$ and $z_3$. 
\begin{center}
 \includegraphics[scale=0.6]{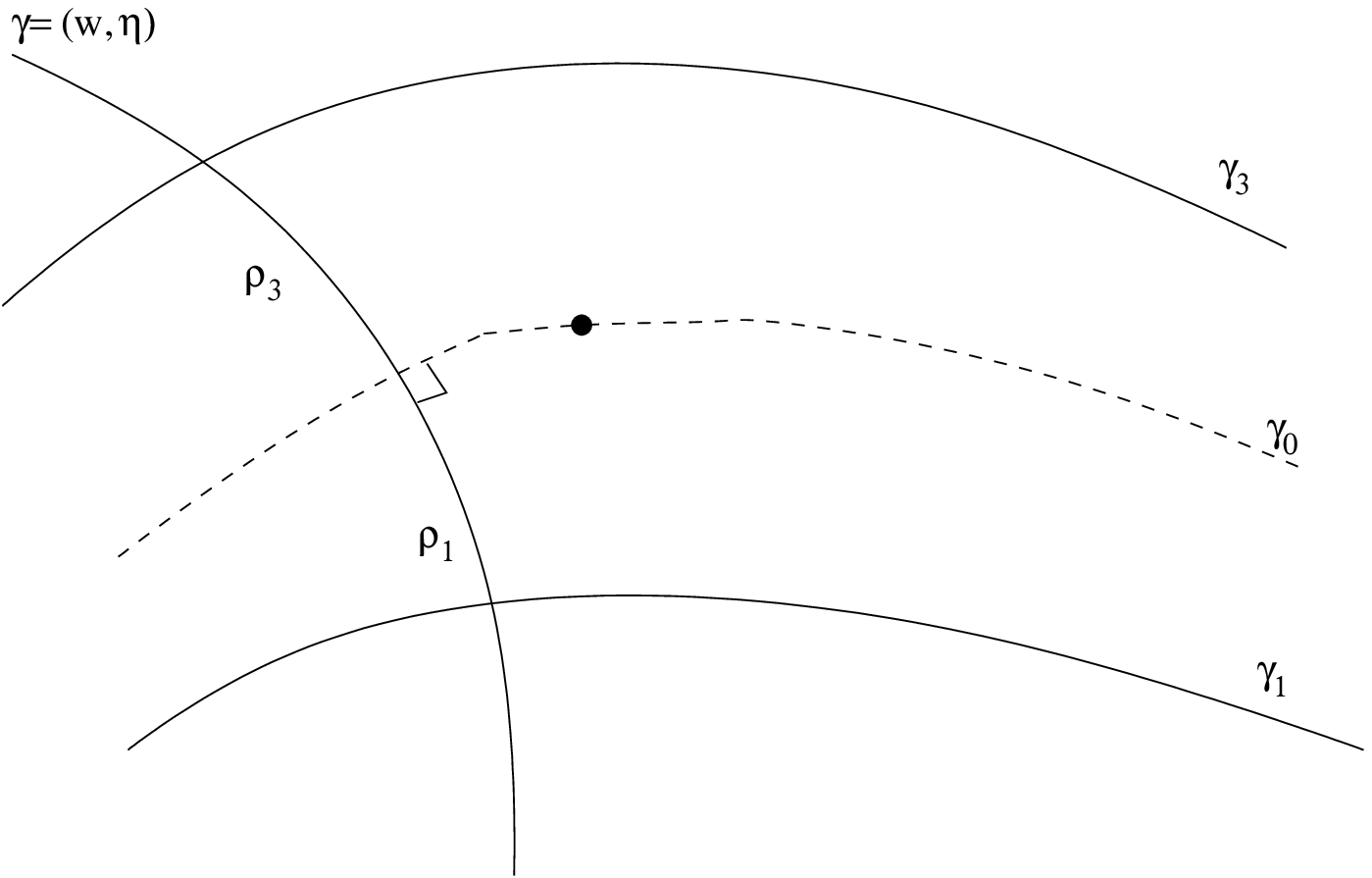}
\end{center}

Fix $\eta$. Then by hyperbolic trigonometry, we have:
\begin{equation*}
 {\rm d}w=-\frac{{\rm d}\rho}{\sinh\rho_1\cot\alpha_1+\sinh\rho_3\cot\alpha_3}.
\end{equation*}
By changing variables, we have:
\begin{equation*}
\mu(\{\gamma\in\mathcal{G}_{13}:\rho(\gamma)\le\rho_0\})=-\frac{1}{2}\int_{\mathcal{G}_D} \delta(\rho_0-\rho)\frac{\cot\alpha_1\cot\alpha_3\sinh\rho\cosh w}{\sinh\rho_1\cot\alpha_1+\sinh\rho_3\cot\alpha_3} \,{\rm d}\rho{\rm d}\eta.
\end{equation*}
which implies:
\begin{equation}
\mu(\{\gamma\in\mathcal{G}_{13}:\rho(\gamma)\le\rho_0\})=\frac{1}{2}\int_{[\eta]}\frac{\cot\alpha_1(\rho_0,\eta)\cot\alpha_3(\rho_0,\eta)\sinh\rho_0\cosh w(\rho_0,\eta)}{\sinh\rho_1(\rho_0,\eta)\cot\alpha_1(\rho_0,\eta)+\sinh\rho_3(\rho_0,\eta)\cot\alpha_3(\rho_0,\eta)} \,{\rm d}\eta.
\end{equation}
where $[\eta]$ is the set of $\eta$ such that there exists $\gamma\in\mathcal{G}_0$ with angle parameter $\eta$ and chord length $\rho_0$.

Let $M_{ij}$ be the chord length distribution with respect to $\gamma_i$ and $\gamma_j$. The above formula yields that:
\begin{equation*}
\int_0^{\rho_0}{\rm d}M_{13}=\frac{1}{2}\int_{[\eta]}\frac{\cot\alpha_1(\rho_0,\eta)\cot\alpha_3(\rho_0,\eta)\sinh\rho_0\cosh w(\rho_0,\eta)}{\sinh\rho_1(\rho_0,\eta)\cot\alpha_1(\rho_0,\eta)+\sinh\rho_3(\rho_0,\eta)\cot\alpha_3(\rho_0,\eta)} \,{\rm d}\eta.
\end{equation*}

By the same argument, we have that:
\begin{equation*}
\int_0^{\rho_0}{\rm d}M_{24}=\frac{1}{2}\int_{[\eta]}\frac{\cot\alpha_2(\rho_0,\eta)\cot\alpha_4(\rho_0,\eta)\sinh\rho_0\cosh w(\rho_0,\eta)}{\sinh\rho_2(\rho_0,\eta)\cot\alpha_2(\rho_0,\eta)+\sinh\rho_4(\rho_0,\eta)\cot\alpha_4(\rho_0,\eta)} \,{\rm d}\eta,
\end{equation*} 

The chord length distributions ${\rm d}M_{12}, {\rm d}M_{23},{\rm d}M_{34}$ and ${\rm d}M_{14}$ have been computed in the former section. Then the chord length distribution $M_Q=(\rho_Q)_*\mu$ for $Q$ is the following:
\begin{eqnarray*}
{\rm d}M_Q(\rho_0)&=&{\rm d}M_{12}({\rho_0})+{\rm d}M_{23}({\rho_0})+{\rm d}M_{34}({\rho_0})+{\rm d}M_{14}({\rho_0})+{\rm d}M_{13}({\rho_0})+{\rm d}M_{24}({\rho_0})\\
&=&\frac{12\rho_0{\rm d}\rho_0}{\sinh^2\rho_0}+{\rm d}M_{13}({\rho_0})+{\rm d}M_{24}({\rho_0}).
\end{eqnarray*}
where ${\rm d}M_{13}$ and ${\rm d}M_{24}$ are described as above. 
\end{proof}

\section{The Pleijel identity in general case}
From the proofs of the theorems and corollaries in this paper, we can see that the most important tool is the hyperbolic trigonometry. This inspires us that this method can be used to prove similar results for $\mathbb{X}_K$, the maximally symmetric, simply connected, $2-$dimensional Riemannian manifold with constant sectional curvature $K$. 

We recall the general sine function for $\mathbb{X}_K$:
\begin{equation}
 \sin_K(x)=x-\frac{Kx^3}{3!}+\frac{K^2x^5}{5!}-\cdots,
\end{equation}
or equivalently by the following formula:
\begin{displaymath}
\sin_K(x) = \left\{ \begin{array}{lll}
\frac{1}{\sqrt{K}}\sin \sqrt{K}x&, & \textrm{if $K>0$,}\\
x &,& \textrm{if $K=0$,}\\
\frac{1}{\sqrt{-K}}\sinh \sqrt{-K}x&, & \textrm{if $K<0$.}
\end{array} \right.
\end{displaymath}

We first give the general sines rules and the general cosine rules. Let $\Delta$ be a geodesic triangle in $\mathbb{X}_K$. Let $A$, $B$ and $C$ be its three angles, and let $a$, $b$ and $c$ be the three edges opposite to $A$, $B$ and $C$ respectively. Then we have:
\begin{eqnarray*}
&&\frac{\sin A}{\sin_K(a)}=\frac{\sin B}{\sin_K(b)}=\frac{\sin C}{\sin_K(c)},\\
&&\cos A=-\cos B\cos C+\sin B\sin C(\sin_K)'(a).
\end{eqnarray*}

Let $\mu_K$ to be the isometry invariant measure on the geodesics set $\mathcal{G}^K$ of $\mathbb{X}_K$. We use $1$ to normalize $\mu_K$ instead of $1/2$ for the Liouville measure in (\ref{c}). Repeat the proof of Theorem \ref{1} by replacing the hyperbolic sine rules by the general ones. Then we obtain the general Ambartzumian-Pleijel identity for $\mathbb{X}_K$:
\begin{theorem}
Let $D$ be a convex compact domain in $\mathbb{X}_K$ with geodesic polygon boundary. Let $f$, $\alpha_1$, $\alpha_2$ and $x$ be the same as in Theorem \ref{1}. Then we have the general Ambartzumian-Pleijel identity:
\begin{equation}
 \int_{\mathcal{G}^K_{D}}(f\circ\rho_K)\,{\rm d}\mu_K=\int_{\mathcal{G}^K_{D}}(f'\circ\rho_K)\sin_K(\rho_K)\cot\alpha_1\cot\alpha_2 \,{\rm d}\mu_K+\sum^n_{i=1}\int^{|a_i|}_0f(x)\,{\rm d}x.
\end{equation}
\end{theorem}
By the same argument as in Theorem \ref{2}, we have the $\mathbb{X}_K$ version of the Pleijel identity:
\begin{theorem}
Let $D$ be a convex compact domain in $\mathbb{X}_K$ with $C^1$ boundary. Let $f$, $\alpha_1$, $\alpha_2$ and $x$ be the same as above. Then we have the following identity:
\begin{equation}
 \int_{\mathcal{G}^K_{D}}(f\circ\rho_K)\,{\rm d}\mu_K=\int_{\mathcal{G}^K_{D}}(f'\circ\rho_K)\sin_K(\rho_K)\cot\alpha_1\cot\alpha_2 \,{\rm d}\mu_K+f(0)L_K(\partial D).
\end{equation}
Moreover if $f(0)=0$, we have the $\mathbb{X}_K$ version of the Pleijel identity.
\end{theorem}
By taking $f$ to be the constant map, we can find that Corollary \ref{c1} still holds for $\mathbb{X}_K$ (without the factor $1/2$). The $\mathbb{X}_K$ version of Corollary \ref{c2} is slightly different from above. By replacing the function $\sinh x$ by $\sin_K(x)$ and $\cosh x$ by $(\sin_K)'(x)$ in the expressions of the volume form of the hyperbolic metric and the Liouville measure, we obtain for $\mathbb{X}_K$ the general isoperimetric inequality:
\begin{equation*} 
L(D)^2\ge4\pi A(D)-KA(D)^2.
\end{equation*}
\bibliographystyle{plain}
\bibliography{Pleijel}
\end{document}